\documentclass{amsart}
\usepackage{amssymb}
\usepackage[abbrev]{amsrefs}
\numberwithin{equation}{section}
\def\N{\mathbb N}
\def\R{\mathbb R}
\def\S{\mathbb S}
\providecommand{\norm}[1]{\left\lVert#1\right\rVert}
\newcommand{\ip}[2]{\left\langle #1, #2 \right\rangle}

\providecommand{\abs}[1]{\left\lvert#1\right\rvert}
\DeclareMathOperator*{\Argmin}{argmin}
\DeclareMathOperator{\Fix}{\mathcal{F}}
\DeclareMathOperator{\AC}{\mathcal{A}}
\DeclareMathOperator{\Diam}{diam}
\newcommand{\sk    }[1]{\left( #1 \right)}
\newcommand{\ck    }[1]{\left\{#1 \right\}}
\newcommand{\CAT}{\textup{CAT}}
\newcommand{\opclitvl}[2]{\left]#1, #2\right]}
\newcommand{\clopitvl}[2]{\left[#1, #2\right[}
\newcommand{\openitvl}[2]{\left]#1, #2\right[}
\theoremstyle{plain}
\newtheorem{theorem}{Theorem}[section]
\newtheorem{lemma}[theorem]{Lemma}

\newtheorem{corollary}[theorem]{Corollary}

\theoremstyle{definition}

\theoremstyle{remark}

\allowdisplaybreaks
\title[Two modified proximal point algorithms] 
{Two modified proximal point algorithms in geodesic spaces with curvature bounded above} 
\author[Y.~Kimura]{Yasunori~Kimura}
\address[Y.~Kimura]
{Department of Information Science, 
Toho University, 
Miyama, Funabashi, Chiba 274-8510, Japan}
\email{yasunori@is.sci.toho-u.ac.jp}
\author[F.~Kohsaka]{Fumiaki~Kohsaka}
\address[F.~Kohsaka]
{Department of Mathematical Sciences, Tokai University, 
Kitakaname, Hiratsuka, Kanagawa 259-1292, Japan}
\email{f-kohsaka@tsc.u-tokai.ac.jp}
\begin{document}
\begin{abstract}
We obtain existence and convergence theorems 
for two variants of the proximal point algorithm 
involving proper lower semicontinuous convex functions 
in complete geodesic spaces with curvature bounded above. 
\end{abstract}
\maketitle

\section{Introduction}
\label{sec:intro}

The aim of this paper is to study the asymptotic behavior of 
sequences generated by two variants of the proximal point algorithm 
for proper lower semicontinuous convex functions 
in admissible complete $\CAT(1)$ spaces. 
We focus not only on the convergence 
of the sequences to minimizers of functions 
but also on the equivalence between 
their boundedness and the existence of minimizers. 
Applications to convex minimization problems 
in complete $\CAT(\kappa)$ spaces  
with a positive real number $\kappa$ 
are also included.  

The proximal point algorithm 
introduced by Martinet~\cite{MR0298899} 
and studied more generally by Rockafellar~\cite{MR0410483} 
is an iterative method for finding zero points of 
maximal monotone operators in Hilbert spaces. 
Bruck and Reich~\cite{MR470761} also obtained some convergence theorems 
for $m$-accretive operators in Banach spaces.  
It is known that this algorithm has a wide range of applications including 
convex minimization problems, variational inequality problems, 
minimax problems, and equilibrium problems. 

For a proper lower semicontinuous convex function 
$f$ of a Hilbert space $H$ into $\opclitvl{-\infty}{\infty}$, 
the proximal point algorithm 
generates a sequence $\{x_n\}$ by $x_1\in H$ and 
\begin{align}\label{eq:ppa-Hilbert} 
 x_{n+1}=J_{\lambda_n f} x_n 
\quad (n=1,2,\dots), 
\end{align}
where $\{\lambda_n\}$ is a sequence of positive real numbers 
and $J_{\lambda_n f}$ is the resolvent of $\lambda_n f$ defined by 
\begin{align*}
 J_{\lambda_n f} x 
= \Argmin_{y\in H} \left\{f(y)+\frac{1}{2\lambda_n}\norm{y-x}^2\right\} 
\end{align*}
for all $n\in \N$ and $x\in H$. 
See also~\cites{MR2798533, MR2548424} for more details 
on convex analysis in Hilbert spaces. 

The celebrated theorem 
by Rockafellar~\cite{MR0410483}*{Theorem~1} implies 
the following existence and weak convergence theorems 
on the sequence $\{x_n\}$ defined by~\eqref{eq:ppa-Hilbert}. 
If $\inf_{n}\lambda_n >0$, then 
$\{x_n\}$ is bounded 
if and only if the set $\Argmin_H f$ 
of all minimizers of $f$ is nonempty. 
Further, in this case, 
$\{x_n\}$ is weakly convergent to an element of $\Argmin_H f$. 
Br{\'e}zis and Lions~\cite{MR491922}*{Th\'eor\`eme~9} 
also showed that $\{x_n\}$ is weakly 
convergent to an element of $\Argmin_H f$ 
if $\sum_{n=1}^{\infty}\lambda_n =\infty$ 
and $\Argmin_H f$ is nonempty. 
Later, G{\"u}ler~\cite{MR1092735}*{Corollary~5.1} 
and Bauschke, Matou{\v{s}}kov{\'a}, and
Reich~\cite{MR2036787}*{Corollary~7.1} 
found the counterexamples to the strong convergence of $\{x_n\}$. 
By assuming the so-called convergence condition,  
Nevanlinna and Reich~\cite{MR0531600}*{Theorem~2}
obtained a strong convergence theorem for 
$m$-accretive operators in Banach spaces. 
In 2000, Solodov and Svaiter~\cite{MR1734665} 
and Kamimura and Takahashi~\cite{MR1788273} 
proposed two different types of strongly 
convergent proximal-type algorithms in Hilbert spaces. 

On the other hand, 
a $\CAT(\kappa)$ space is a geodesic metric space 
such that every geodesic triangle in it 
satisfies the $\CAT(\kappa)$ inequality, 
where $\kappa$ is a real number. 
A complete $\CAT(0)$ space is particularly 
called an Hadamard space. 
Since the concept of 
$\CAT(\kappa)$ spaces 
includes several fundamental spaces, 
the fixed point theory and the convex optimization theory 
in such spaces have been increasingly 
important. See, for instance,~\cites{MR3241330, MR1744486, MR1072312} 
for more details in this direction.   

In the 1990s, 
Jost~\cite{MR1360608} and
Mayer~\cite{MR1651416} 
generalized the concept of resolvents of convex functions 
to Hadamard spaces. 
According to~\cite{MR3241330}*{Section~2.2},~\cite{MR1360608}*{Lemma~2}, 
and~\cite{MR1651416}*{Section~1.3}, 
if $f$ is a proper lower semicontinuous convex 
function of an Hadamard space $X$ into $\opclitvl{-\infty}{\infty}$, 
then the resolvent $J_{f}$ of $f$ given by 
\begin{align}\label{eq:resolvent-Hadamard}
J_{f} x = \Argmin_{y\in X} \left\{f(y)+\frac{1}{2}d(y, x)^2\right\} 
\end{align}
for all $x\in X$ is a single-valued nonexpansive mapping of $X$ into itself. 
In this case, the set $\Fix(J_f)$ of 
all fixed points of $J_f$ coincides with $\Argmin_X f$. 
See also~\cite{MR3241330} 
on convex analysis in Hadamard spaces. 

In 2013, 
Ba{\v{c}}{\'a}k~\cite{MR3047087} 
generalized the classical theorem by 
Br{\'e}zis and Lions~\cite{MR491922}*{Th\'eor\`eme~9} 
to Hadamard spaces. 
Some related asymptotic results were found by 
Ariza-Ruiz, Leu\c stean, and 
L\'opez-Acedo~\cite{MR3206460}*{Corollary~6.6} 
and Ba{\v{c}}{\'a}k and Reich~\cite{MR3346750}*{Proposition~1.5}. 
\begin{theorem}[\cite{MR3047087}*{Theorem~1.4}]
\label{thm:Bacak} 
 Let $X$ be an Hadamard space, 
 $f$ a proper lower semicontinuous convex function 
 of $X$ into $\opclitvl{-\infty}{\infty}$,  
 $J_{\eta f}$ the resolvent of $\eta f$ for all $\eta >0$, 
 and $\{x_n\}$ a sequence defined by 
 $x_1\in X$ and~\eqref{eq:ppa-Hilbert}, 
 where $\{\lambda_n\}$ is a sequence of positive real numbers  
 such that $\sum_{n=1}^{\infty} \lambda_n=\infty$. 
 If $\Argmin_X f$ is nonempty, 
 then $\{x_n\}$ is $\Delta$-convergent 
 to an element of $\Argmin_X f$. 
\end{theorem}

Motivated by~\cites{MR2780284, MR3047087, MR1788273}, 
the authors~\cite{MR3574140} 
recently obtained the following existence and convergence theorems 
for two variants  of the proximal point algorithm in Hadamard spaces. 
The algorithms~\eqref{eq:Mann-CAT0} 
and~\eqref{eq:Halpern-CAT0}  
were originally introduced by 
Eckstein and Bertsekas~\cite{MR1168183} 
and Kamimura and Takahashi~\cite{MR1788273} 
for maximal monotone operators 
in Hilbert spaces, respectively. 
\begin{theorem}[\cite{MR3574140}*{Theorem~4.2}]
\label{thm:Mann-CAT0}
 Let $X$, $f$, $\{J_{\eta f}\}_{\eta>0}$ 
 be the same as in Theorem~\ref{thm:Bacak} 
 and $\{x_n\}$ a sequence defined by $x_1\in X$ and 
 \begin{align}\label{eq:Mann-CAT0}
  x_{n+1} = \alpha_n x_n \oplus (1-\alpha_n) J_{\lambda_n f}x_n 
 \quad (n=1,2,\dots), 
 \end{align}
 where $\{\alpha_n\}$ is a sequence in $\clopitvl{0}{1}$ 
 and $\{\lambda_n\}$ is a sequence of positive real numbers  
 such that $\sum_{n=1}^{\infty}(1-\alpha_n) \lambda_n=\infty$.  
 Then the following hold. 
 \begin{enumerate}
  \item[(i)] The set $\Argmin_X f$ is nonempty 
 if and only if 
 $\{J_{\lambda_n f}x_n\}$ is bounded; 
  \item[(ii)] if $\Argmin_X f$ is nonempty and $\sup_n\alpha_n<1$, 
 then both $\{x_n\}$ and $\{J_{\lambda_n f}x_n\}$ 
 are $\Delta$-convergent to an element $x_{\infty}$ of $\Argmin_X f$. 
 \end{enumerate} 
\end{theorem}

\begin{theorem}[\cite{MR3574140}*{Theorem~5.1}]
\label{thm:Halpern-CAT0}
 Let $X$, $f$, and $\{J_{\eta f}\}_{\eta>0}$ be the same as in 
 Theorem~\ref{thm:Bacak}, 
 $v$ an element of $X$,  
 and $\{y_n\}$ a sequence defined by $y_1\in X$ and 
 \begin{align}\label{eq:Halpern-CAT0}
  y_{n+1} = \alpha_n v \oplus (1-\alpha_n) J_{\lambda_n f}y_n 
 \quad (n=1,2,\dots), 
 \end{align}
 where $\{\alpha_n\}$ is a sequence in $[0,1]$ and 
 $\{\lambda_n\}$ is a sequence of positive real numbers  
 such that $\lim_{n} \lambda_n=\infty$. 
 Then the following hold. 
 \begin{enumerate}
  \item[(i)] The set $\Argmin_X f$ is nonempty 
 if and only if 
 $\{J_{\lambda_n f}y_n\}$ is bounded; 
  \item[(ii)] if $\Argmin_X f$ is nonempty, $\lim_n \alpha_n = 0$, 
 and $\sum_{n=1}^{\infty}\alpha_n=\infty$, 
 then both $\{y_n\}$ and $\{J_{\lambda_n f}y_n\}$ 
 are convergent to $Pv$, 
 where $P$ denotes the metric projection 
 of $X$ onto $\Argmin_X f$. 
 \end{enumerate} 
\end{theorem}

\begin{theorem}[\cite{MR3574140}*{Theorem~5.4}]
\label{thm:Halpern-CAT0-another}
 Let $X$, $f$, and $\{J_{\eta f}\}_{\eta>0}$ be the same as in 
 Theorem~\ref{thm:Bacak}, 
 $v$ an element of $X$,  
 and $\{y_n\}$ a sequence defined by $y_1\in X$ 
 and~\eqref{eq:Halpern-CAT0}, 
 where $\{\alpha_n\}$ is a sequence in $\opclitvl{0}{1}$ 
 and $\{\lambda_n\}$ is a sequence of positive real numbers  
 such that 
 \begin{align*}
  \lim_{n\to \infty} \alpha_n = 0, 
  \quad \sum_{n=1}^{\infty}\alpha_n=\infty, 
  \quad \textrm{and} \quad 
   \inf_n \lambda_n > 0.   
 \end{align*}
 If $\Argmin_X f$ is nonempty, 
 then both $\{y_n\}$ and $\{J_{\lambda_n f}y_n\}$ 
 are convergent to $Pv$, 
 where $P$ denotes the metric projection 
 of $X$ onto $\Argmin_X f$. 
\end{theorem}

In 2015, 
Ohta and P{\'a}lfia~\cite{MR3396425}*{Definition~4.1 and Lemma~4.2}
showed that the resolvent $J_{f}$ given by~\eqref{eq:resolvent-Hadamard} 
is still well defined in a complete $\CAT(1)$ space 
such that $\Diam(X)<\pi/2$, where $\Diam(X)$ 
denotes the diameter of $X$. 
Using this result, they~\cite{MR3396425}*{Theorem~5.1} 
obtained a $\Delta$-convergence theorem 
on the proximal point algorithm in such spaces. 
It should be noted that 
the condition that $\Diam(X)<\pi/2$ for a complete $\CAT(1)$ space $X$ 
corresponds to the boundedness condition 
for an Hadamard space. 
In fact, every sequence 
in a complete $\CAT(1)$ space $X$ such that 
$\Diam(X)<\pi/2$ has a $\Delta$-convergent subsequence.  

In 2016, the authors~\cite{MR3463526} introduced 
another type of resolvents of convex functions 
in $\CAT(1)$ spaces. 
For a given proper lower semicontinuous 
convex function $f$ of an 
admissible complete $\CAT(1)$ space $X$ 
into $\opclitvl{-\infty}{\infty}$, 
they~\cite{MR3463526}*{Definition~4.3}  
defined the resolvent $R_{f}$ of $f$ by 
\begin{align}\label{eq:resolvent-CAT1-def2}
 R_{f}x = \Argmin_{y\in X} \bigl\{f(y) + \tan d(y, x) \sin d(y, x)\bigr\} 
\end{align}
for all $x\in X$. 
Following~\cite{MR3638673}, 
we say that a $\CAT(1)$ space $X$ is admissible
\begin{align}\label{eq:admissible}
 d(w,w')< \frac{\pi}{2} 
\end{align}
for all $w,w'\in X$. 

Recently, the authors~\cite{MR3638673} 
obtained the following result 
on the proximal point algorithm in $\CAT(1)$ spaces.  
We note that the $\Delta$-convergence of $\{x_n\}$ 
in Theorem~\ref{thm:KimuraKohsaka-PPA-CAT1} 
was also found independently by 
Esp{\'{\i}}nola and
Nicolae~\cite{EspinolaNicolae-JNCA16}*{Theorem~3.2}. 
\begin{theorem}[\cite{MR3638673}*{Theorems~1.1 and~1.2}]
 \label{thm:KimuraKohsaka-PPA-CAT1}
 Let $X$ be an admissible complete $\CAT(1)$ space, 
 $f$ a proper lower semicontinuous convex function 
 of $X$ into $\opclitvl{-\infty}{\infty}$,  
 $R_{\eta f}$ the resolvent of $\eta f$ for all $\eta >0$, 
 and $\{x_n\}$ a sequence defined by $x_1\in X$ and 
 \begin{align*}
  x_{n+1} = R_{\lambda_n f} x_n 
 \quad (n=1,2,\dots), 
 \end{align*}
 where $\{\lambda_n\}$ is a sequence of positive real numbers  
 such that 
 $\sum_{n=1}^{\infty}\lambda_n =\infty$. 
 Then $\Argmin_X f$ is nonempty if and only if 
 \begin{align*}
  \inf_{y\in X}\limsup_{n\to \infty}d(y, x_n)< \frac{\pi}{2} 
  \quad \textrm{and} \quad 
  \sup_{n} d(x_{n+1}, x_n) < \frac{\pi}{2}.    
 \end{align*} 
 Further, in this case, 
 $\{x_n\}$ is $\Delta$-convergent to an element of $\Argmin_X f$. 
\end{theorem}

Motivated by the papers mentioned above, 
we study the asymptotic behavior of sequences generated by 
$x_1, y_1, v\in X$, 
 \begin{align}\label{eq:Mann-CAT1}
  x_{n+1} = \alpha_n x_n \oplus (1-\alpha_n) R_{\lambda_n f}x_n 
 \quad (n=1,2,\dots), 
 \end{align}
and 
 \begin{align}\label{eq:Halpern-CAT1}
  y_{n+1} = \alpha_n v \oplus (1-\alpha_n) R_{\lambda_n f}y_n 
 \quad (n=1,2,\dots), 
 \end{align}
where $\{\alpha_n\}$ is a sequence in $[0,1]$, 
$\{\lambda_n\}$ is a sequence of positive real numbers, 
$X$ is an admissible complete $\CAT(1)$ space, 
$f$ is a proper lower semicontinuous convex function 
of $X$ into $\opclitvl{-\infty}{\infty}$, 
and $R_{\lambda_n f}$ is the resolvent of $\lambda_n f$ 
for all $n\in \N$ 
given by~\eqref{eq:resolvent-CAT1-def2}. 
These algorithms correspond to~\eqref{eq:Mann-CAT0} 
and~\eqref{eq:Halpern-CAT0}  
in Hadamard spaces, respectively. 

This paper is organized as follows. 
In Section~\ref{sec:pre}, we recall some definitions 
and results needed in this paper. 
In Section~\ref{sec:res}, 
we obtain some fundamental properties of resolvents 
of convex functions in complete $\CAT(1)$ spaces. 
In Sections~\ref{sec:Mann} and~\ref{sec:Halpern}, 
we study the asymptotic behavior of 
sequences generated by~\eqref{eq:Mann-CAT1} 
and~\eqref{eq:Halpern-CAT1}, respectively. 
The three main results in this paper,  
Theorems~\ref{thm:Mann-CAT1},~\ref{thm:Halpern-CAT1}, 
and~\ref{thm:Halpern-CAT1-another} 
in admissible complete $\CAT(1)$ spaces,  
correspond 
to Theorems~\ref{thm:Mann-CAT0},~\ref{thm:Halpern-CAT0}, 
and~\ref{thm:Halpern-CAT0-another} 
in Hadamard spaces, respectively. 
In Section~\ref{sec:cor}, 
we deduce three corollaries of our results  
in complete $\CAT(\kappa)$ spaces 
with a positive real number $\kappa$. 

\section{Preliminaries}
\label{sec:pre}

Throughout this paper, 
we denote by $\N$ the set of all positive integers, 
$\R$ the set of all real numbers, 
$\opclitvl{-\infty}{\infty}$ the set $\R \cup \{\infty\}$, 
$\R^2$ the two dimensional Euclidean space with 
Euclidean metric $\rho_{\R^2}$, 
$\S^2$ the unit sphere of the three dimensional Euclidean space $\R^3$ 
with the spherical metric $\rho_{\S^2}$,  
$H$ a real Hilbert space with inner product $\ip{\,\cdot\,}{\,\cdot\,}$  
and induced norm $\norm{\,\cdot\,}$, 
$X$ a metric space with metric $d$, 
$\Fix(T)$ the set of all fixed points of 
a mapping $T$ of $X$ into itself, 
and $\Argmin_X f$ or $\Argmin_{y\in X}f(y)$ 
the set of all minimizers of 
a function $f$ of $X$ into $\opclitvl{-\infty}{\infty}$. 
In the case where $\Argmin_X f$ is a singleton $\{p\}$, 
we sometimes identify $\Argmin_X f$ with the single point $p$. 

We need the following lemma. 

\begin{lemma}[\cite{MR1911872}*{Lemma~2.5}; 
see also~\cite{MR2338104}*{Lemma~2.3}]\label{lem:AKTT-seq}
 Let $\{s_n\}$ be a sequence of nonnegative real numbers, 
 $\{\alpha_n\}$ a sequence in $[0,1]$ such that 
 $\sum_{n=1}^{\infty}\alpha_n=\infty$, 
 and $\{t_n\}$ a sequence of real numbers 
 such that $\limsup_{n}t_n \leq 0$. 
 Suppose that 
 \begin{align}\label{eq:lem:AKTT}
  s_{n+1} \leq (1-\alpha_n) s_n + \alpha_n t_n
 \end{align}
 for all $n\in \N$. Then $\lim_{n}s_n=0$. 
\end{lemma}

Saejung and Yotkaew~\cite{MR2847453} 
found the following variant of Lemma~\ref{lem:AKTT-seq}. 
Later, Kimura and Saejung~\cite{MR3570781} 
filled in a slight gap in the original proof of this result. 
Although it was assumed in~\cites{MR3570781, MR2847453} that 
$\alpha_n<1$ for all $n\in \N$, 
the proof in~\cite{MR3570781}*{Lemma~2.8} 
is also valid in the case below. 

\begin{lemma}[\cite{MR3570781}*{Lemma~2.8} 
 and~\cite{MR2847453}*{Lemma~2.6}]\label{lem:KSY-seq}
 Let $\{s_n\}$ be a sequence of nonnegative real numbers, 
 $\{\alpha_n\}$ a sequence in $\opclitvl{0}{1}$ such that 
 $\sum_{n=1}^{\infty}\alpha_n=\infty$, 
 and $\{t_n\}$ a sequence of real numbers. 
 Suppose that~\eqref{eq:lem:AKTT} holds for all $n\in \N$ 
 and that $\limsup_{i}t_{n_i} \leq 0$ 
 whenever 
 $\{n_i\}$ is an increasing sequence in $\N$ 
 satisfying 
 \begin{align*}
  \limsup_{i\to \infty} \bigl(s_{n_i}-s_{n_i + 1}\bigr) \leq 0. 
 \end{align*}
 Then $\lim_{n}s_n=0$. 
\end{lemma}

Let $\kappa$ be a nonnegative real number and 
$D_{\kappa}$ the extended real number defined by 
$D_{\kappa}=\infty$ if $\kappa =0$ 
and $\pi/{\sqrt{\kappa}}$ if $\kappa >0$. 
A metric space $X$ is said to be $D_{\kappa}$-geodesic 
if for each $x,y\in X$ with $d(x,y) < D_{\kappa}$, 
there exists a mapping 
$c$ of $[0,l]$ into $X$ such that 
$c(0)=x$, $c(l)=y$, 
and 
\begin{align*}
 d\bigl(c(t_1), c(t_2)\bigr)=\abs{t_1-t_2}
\end{align*}
for all $t_1,t_2\in [0,l]$, 
where $l=d(x,y)$. 
The mapping $c$ is called a geodesic path from $x$ to $y$. 
The image of $c$ is denoted by $[x,y]_c$ and 
is called a geodesic segment between $x$ and $y$. 
We denote by $\alpha x\oplus_c (1-\alpha) y$ 
the point given by 
\begin{align*}
 \alpha x\oplus_c (1-\alpha) y = c\bigl((1-\alpha) l\bigr)
\end{align*}
for all $\alpha \in [0,1]$. 
A $D_{\kappa}$-geodesic metric space 
is also called a $D_{\kappa}$-geodesic space. 
An $\infty$-geodesic metric space is also 
called a geodesic metric space or a geodesic space. 
A subset $F$ of a $D_{\kappa}$-geodesic space $X$ 
such that $d(w,w')<D_{\kappa}$ for all $w,w'\in F$ 
is said to be convex if $[x,y]_c \subset F$ 
whenever $x,y\in F$ 
and $c$ is a geodesic path from $x$ to $y$. 
Although $[x,y]_c$ and $\alpha x \oplus_c (1-\alpha) y$ 
depend on the choice of a geodesic path $c$ from $x$ to $y$, 
we sometimes denote them 
simply by $[x,y]$ and $\alpha x \oplus (1-\alpha)y$, respectively. 
They are determined uniquely 
if the space $X$ is uniquely $D_{\kappa}$-geodesic, that is, 
for each $x,y\in X$ with $d(x,y) <D_{\kappa}$, 
there exists a unique geodesic path from $x$ to $y$. 

If $H$ is a Hilbert space, then the unit sphere $S_H$ of $H$ 
is a uniquely $\pi$-geodesic complete metric space 
with the spherical metric $\rho_{S_H}$ defined by 
\begin{align*}
 \rho _{S_H} (x, y) = \arccos \ip{x}{y}
\end{align*}
for all $x,y\in S_H$. 
For all distinct $x,y\in S_H$ with $\rho_{S_H}(x,y)<\pi$, 
the unique geodesic path $c$ from $x$ to $y$ 
is given by 
\begin{align*}
 c(t) = (\cos t) x + (\sin t) \cdot \frac{y-\ip{x}{y}x}{\norm{y-\ip{x}{y}x}}
\end{align*}
for all $t\in [0,\rho_{S_H}(x,y)]$. 
The space $(S_H, \rho_{S_H})$ is called a Hilbert sphere. 
See~\cites{MR3706153, MR1744486, MR744194} for more details 
on Hilbert spheres. 

Let $(M_{\kappa}, d_{\kappa})$ be 
the uniquely $D_{\kappa}$-geodesic space given by 
\begin{align*}
 (M_{\kappa}, d_{\kappa}) 
 = 
 \begin{cases}
  \left(\R^2, \rho_{\R^2}\right) & (\kappa=0); \\
  \left(\S^2, \frac{1}{\sqrt{\kappa}}\rho_{\S^2}\right) & (\kappa>0). 
 \end{cases}
\end{align*}
If $\kappa$ is a nonnegative real number, 
$X$ is a $D_{\kappa}$-geodesic space, 
and $x_1,x_2,x_3$ are points of $X$ satisfying 
\begin{align}\label{eq:three-point}
 d(x_1,x_2)+d(x_2,x_3)+d(x_3,x_1)<2D_{\kappa}, 
\end{align}
then there exist $\bar{x}_1, \bar{x}_2, \bar{x}_3\in M_{\kappa}$ 
such that 
\begin{align*}
 d(x_i, x_j)=d_{\kappa}(\bar{x}_i, \bar{x}_j)
\end{align*}
for all $i,j\in \{1,2,3\}$; 
see~\cite{MR1744486}*{Lemma~2.14 in Chapter~I.2}.  
The two sets $\Delta$ and $\bar{\Delta}$ given by 
\begin{align*}
 \Delta = [x_1,x_2] \cup [x_2, x_3] \cup [x_3, x_1] 
 \quad \textrm{and} \quad 
 \bar{\Delta} 
 = [\bar{x}_1,\bar{x}_2] \cup [\bar{x}_2, \bar{x}_3] \cup [\bar{x}_3, \bar{x}_1] 
\end{align*}
are called a geodesic triangle with vertices $x_1,x_2,x_3$ 
in $X$ and a comparison triangle for $\Delta$, 
respectively. 
A point $\bar{p}\in \bar{\Delta}$ is called a comparison point 
for $p\in \Delta$ if 
\begin{align*}
 p \in [x_i, x_j], \quad \bar{p} \in [\bar{x}_i, \bar{x}_j], 
 \quad \textrm{and} \quad 
 d(x_i, p) = d_{\kappa} (\bar{x}_i, \bar{p}) 
\end{align*}
for some distinct $i, j\in \{1,2,3\}$. 
A metric space $X$ is said to be 
a $\CAT(\kappa)$ space if it is $D_{\kappa}$-geodesic and 
the $\CAT(\kappa)$ inequality 
\begin{align*}
 d(p, q) \leq d_{\kappa}(\bar{p}, \bar{q})
\end{align*}
holds whenever $\Delta$ is a geodesic triangle 
with vertices $x_1, x_2, x_3 \in X$ satisfying~\eqref{eq:three-point}, 
$\bar{\Delta}$ is a comparison triangle for $\Delta$, 
and $\bar{p}, \bar{q}\in \bar{\Delta}$ are 
comparison points for $p, q\in \Delta$, respectively. 
In this case, the space $X$ 
is also uniquely $D_{\kappa}$-geodesic. 
Every $\CAT(\kappa)$ space 
is a $\CAT(\kappa')$ space 
for all $\kappa'\in \openitvl{\kappa}{\infty}$. 
A complete $\CAT(0)$ space is particularly called 
an Hadamard space. 
The class of Hadamard spaces 
includes nonempty closed convex subsets of Hilbert spaces, 
open unit balls of Hilbert spaces with hyperbolic metric, 
Hadamard manifolds, and complete $\R$-trees. 
The class of complete $\CAT(1)$ spaces 
includes Hadamard spaces and Hilbert spheres with 
spherical metric. 
We say that a $\CAT(1)$ space $X$ is 
admissible if~\eqref{eq:admissible} holds for all $w,w'\in X$. 
If $\kappa>0$, then 
$(X,d)$ is a complete $\CAT(\kappa)$ space 
such that $d(w,w')<D_{\kappa}/2$ for all $w,w'\in X$ 
if and only if $(X, \sqrt{\kappa} d)$ is 
an admissible complete $\CAT(1)$ space. 
See~\cites{MR3241330, MR1744486, MR1835418} for more details 
on $\CAT(\kappa)$ spaces. 

We know that if $X$ is a $\CAT(1)$ space, 
$x_1,x_2,x_3\in X$ satisfy~\eqref{eq:three-point} 
for $\kappa =1$, 
and $\alpha\in [0,1]$, then 
\begin{align}\label{eq:CAT1-ineq}
 \cos d\bigl(\alpha x_1 \oplus (1-\alpha) x_2, x_3\bigr) 
 \geq \alpha \cos d(x_1, x_3) + (1-\alpha) \cos d(x_2, x_3). 
\end{align}
We also know the following fundamental inequalities. 
\begin{lemma}[\cite{MR2927571}*{Corollary~2.2}]\label{lem:KS-ineq}
 If $X$ is a $\CAT(1)$ space, 
 $x_1,x_2,x_3 \in X$ satisfy~\eqref{eq:three-point} for $\kappa =1$, 
and $\alpha\in [0,1]$, then 
\begin{align}
 \begin{split}\label{eq:KS-ineq}
 &\cos d\bigl(\alpha x_1 \oplus (1-\alpha) x_2, x_3\bigr)\sin d(x_1,x_2) \\
 &\geq \cos d(x_1, x_3) \sin \bigl(\alpha d(x_1,x_2)\bigr) 
 + \cos d(x_2, x_3) \sin \bigl((1-\alpha)d(x_1,x_2)\bigr).   
 \end{split}
\end{align} 
\end{lemma}

\begin{lemma}[\cite{MR3020188}*{Lemma~3.1}]\label{lem:KS-ineq-Halpern}
 If $X$ is an admissible $\CAT(1)$ space, $x_1,x_2,x_3\in X$, 
 and $\alpha\in [0,1]$, then 
\begin{align}
 \begin{split}\label{eq:KS-ineq-Halpern}
 &\cos d\bigl(\alpha x_1 \oplus (1-\alpha) x_2, x_3\bigr) \\
 &\geq (1-\beta) \cos d(x_2, x_3) 
 + \beta \cdot 
 \frac{\cos d(x_1, x_3)}{\sin d(x_1, x_2) 
\tan \bigl(\frac{\alpha}{2}d(x_1, x_2)\bigr) + \cos d(x_1,x_2)}, 
 \end{split}
\end{align}
where 
\begin{align*}
 \beta
 =
 \begin{cases}
  1 - \dfrac{\sin \bigl((1-\alpha) d(x_1,x_2)\bigr)}{\sin d(x_1,x_2)} 
 & (x_1\neq x_2); \\
  \alpha & (x_1=x_2). 
 \end{cases}
\end{align*}
\end{lemma}

The concept of $\Delta$-convergence was originally 
introduced by Lim~\cite{MR423139} in metric spaces. 
Later, Kirk and Panyanak~\cite{MR2416076} 
applied it to the study of geodesic spaces. 
Let $X$ be a metric space 
and $\{x_n\}$ a sequence in $X$. 
The asymptotic center $\AC\bigl(\{x_n\}\bigr)$ 
of $\{x_n\}$ is defined by 
\begin{align*}
 \AC\bigl(\{x_n\}\bigr) 
 =\ck{z\in X: \limsup_{n\to \infty} d(z, x_n) 
 =\inf_{y\in X} \limsup_{n\to \infty} d(y, x_n)}. 
\end{align*}
The sequence $\{x_n\}$ is said to be 
$\Delta$-convergent to $p\in X$ if 
\begin{align*}
 \AC\bigl(\{x_{n_i}\}\bigr)=\{p\}
\end{align*}
holds for each subsequence $\{x_{n_i}\}$ of $\{x_n\}$.   
In this case, $\{x_n\}$ is bounded and its each subsequence is 
also $\Delta$-convergent to $p$. 
If $X$ is a nonempty closed convex subset of a Hilbert space, 
then the $\Delta$-convergence coincides with the weak convergence. 
We denote by $\omega_{\Delta}\bigl(\{x_n\}\bigr)$ 
the set of all points $q\in X$ such that 
there exists a subsequence of $\{x_n\}$ 
which is $\Delta$-convergent to $q$. 
Following~\cite{MR3638673}, 
we say that a sequence $\{x_n\}$ in a $\CAT(1)$ space $X$ 
is spherically bounded if 
\begin{align*}
\inf_{y\in X}\limsup_{n\to \infty}d(y, x_n)< \frac{\pi}{2}.  
\end{align*}
We know the following lemmas. 
\begin{lemma}[\cite{MR2508878}*{Proposition~4.1 and Corollary~4.4}]
 \label{lem:EFL-subseq}
 Let $X$ be a complete $\CAT(1)$ space 
 and $\{x_n\}$ a spherically bounded sequence in $X$. 
 Then $\AC\bigl(\{x_n\}\bigr)$ is a singleton 
 and $\{x_n\}$ has a $\Delta$-convergent subsequence. 
\end{lemma}
 
\begin{lemma}[\cite{MR3213144}*{Proposition~3.1}]\label{lem:KSY-conv}
 Let $X$ be a complete $\CAT(1)$ space 
 and $\{x_n\}$ a spherically bounded sequence in $X$. 
 If $\{d(z, x_n)\}$ is convergent for each 
 element $z$ of $\omega_{\Delta}\bigl(\{x_n\}\bigr)$, 
 then $\{x_n\}$ is $\Delta$-convergent to an element of $X$.  
\end{lemma}

Let $X$ be an admissible $\CAT(1)$ space. 
A function $f$ of $X$ into $\opclitvl{-\infty}{\infty}$ 
is said to be proper if $f(a) \in \R$ for some $a\in X$. 
It is also said to be convex if 
\begin{align*}
 f\bigl(\alpha x \oplus (1-\alpha) y\bigr) 
 \leq \alpha f(x) + (1-\alpha) f(y)
\end{align*}
whenever $x,y\in X$ and $\alpha \in \openitvl{0}{1}$. 
We denote by $\mathit{\Gamma}_0(X)$ the set of all 
proper lower semicontinuous convex functions 
of $X$ into $\opclitvl{-\infty}{\infty}$. 
The set $\Argmin_X f$ is obviously closed and convex 
for each $f\in \mathit{\Gamma}_0(X)$. 
It follows from~\eqref{eq:CAT1-ineq} that 
$-\cos d(\cdot, z)$ belongs to $\mathit{\Gamma}_0(X)$ 
for all $z\in X$. 
For a nonempty closed convex subset $C$ of $X$, 
the indicator function $i_C$ for $C$, which is defined by 
$i_C(x)=0$ if $x\in C$ and $\infty$ otherwise, 
belongs to $\mathit{\Gamma}_0(X)$.  
See~\cites{MR1113394, MR3523548} on convex functions 
in $\CAT(1)$ spaces. 
A function $f$ of $X$ into $\opclitvl{-\infty}{\infty}$ 
is said to be $\Delta$-lower semicontinuous if 
$f(p) \leq \liminf_{n} f(x_n)$ 
whenever $\{x_n\}$ is a sequence in $X$ 
which is $\Delta$-convergent to $p\in X$. 
A function $g$ of $X$ into $\clopitvl{-\infty}{\infty}$ 
is said to be concave if $-g$ is convex. 

Let $X$ be an admissible complete $\CAT(1)$ space 
and $f$ an element of $\mathit{\Gamma}_0(X)$. 
It is known~\cite{MR3463526}*{Theorem~4.2} that 
for each $x\in X$, 
there exists a unique $\hat{x}\in X$ 
such that 
\begin{align*}
 f(\hat{x})+\tan d(\hat{x},x)\sin d(\hat{x}, x) 
 = \inf_{y\in X} \left\{
 f(y) + \tan d(y, x) \sin d(y, x)\right\}. 
\end{align*}
Following~\cite{MR3463526}*{Definition~4.3}, 
we define the resolvent $R_{f}$ of $f$ by 
\begin{align*}
 R_{f}x =\hat{x} 
\end{align*}
for all $x\in X$. 
In other words, $R_{f}$ can be defined 
by~\eqref{eq:resolvent-CAT1-def2} for all $x\in X$. 
If $f$ is the indicator function $i_C$ 
for a nonempty closed convex subset $C$ of $X$, 
then the resolvent $R_{f}$ coincides with the metric projection $P_C$ 
of $X$ onto $C$, that is, 
\begin{align*}
 R_{f}x
 =\Argmin_{y\in C} \tan d(y, x) \sin d(y, x) 
 =\Argmin_{y\in C} d(y, x) = P_Cx  
\end{align*}
for all $x\in X$. 

It is known~\cite{MR3463526}*{Theorems~4.2 and~4.6} that 
$R_{f}$ is a well-defined and single-valued mapping 
of $X$ into itself, 
\begin{align}\label{eq:resolvent-CAT1-fix}
 \Fix(R_{f})=\Argmin_X f,  
\end{align} 
and 
\begin{align}
 \begin{split}\label{eq:resolvent-CAT1-firm-org}
 &\bigl(
 C_{x}^2 (1+C_{y}^2)C_{y}
 + 
  C_{y}^2 (1+C_{x}^2)C_{x}
 \bigr)
 \cos d(R_{f}x, R_{f}y) \\
 &\geq 
 C_{x}^2 (1+C_{y}^2)
 \cos d(R_{f}x, y)
 +
 C_{y}^2 (1+C_{x}^2)
 \cos d(R_{f}y, x)
 \end{split}
\end{align} 
for all $x,y\in X$, 
where $C_{z}$ the real number 
given by $C_{z}=\cos d(R_{f}z, z)$ 
for all $z\in X$. 

The following lemma was 
recently obtained 
in~\cite{MR3638673}*{Lemma~3.1 and Corollary~3.2}.  
The inequality~\eqref{eq:resolvent-CAT1-firm} 
is a generalization of~\eqref{eq:resolvent-CAT1-firm-org} 
and corresponds to~\cite{MR2780284}*{Lemma~3.1} 
in Banach spaces.   

\begin{lemma}[\cite{MR3638673}*{Lemma~3.1 and
 Corollary~3.2}]
 \label{lem:KK-resolvent-ineq}
 Let $X$ be an admissible complete $\CAT(1)$ space, 
 $f$ an element of $\mathit{\Gamma}_0(X)$, 
 $R_{\eta f}$ the resolvent of $\eta f$ for all $\eta >0$, 
 and $C_{\eta, z}$ the real number given by 
\begin{align}\label{eq:C-def}
 C_{\eta, z}=\cos d(R_{\eta f}z, z)
\end{align}
for all $\eta>0$ and $z\in X$.  
Then 
\begin{align}
 \begin{split}\label{eq:resolvent-CAT1-firm}
 &\bigl(
 \lambda C_{\lambda, x}^2 (1+C_{\mu, y}^2)C_{\mu, y}
 + 
  \mu C_{\mu, y}^2 (1+C_{\lambda, x}^2)C_{\lambda, x}
 \bigr)
 \cos d(R_{\lambda f}x, R_{\mu f}y) \\
 &\geq 
 \lambda C_{\lambda, x}^2 (1+C_{\mu, y}^2)
 \cos d(R_{\lambda f}x, y)
 +
 \mu C_{\mu, y}^2 (1+C_{\lambda, x}^2)
 \cos d(R_{\mu f}y, x)
 \end{split}
\end{align} 
holds for all $x,y\in X$ and $\lambda, \mu>0$. 
Further, 
\begin{align}
 \begin{split}\label{eq:resolvent-CAT1-sqfirm}
 &\frac{\pi}{2}\sk{
 \frac{1}{C_{\lambda, x}^2}+1} 
 \bigl(
 C_{\lambda, x} \cos d(u, R_{\lambda f}x) - \cos d(u, x) 
 \bigr)
 \geq 
 \lambda \bigl(f(R_{\lambda f}x) - f(u)\bigr) 
 \end{split}
\end{align} 
and 
\begin{align}\label{eq:resolvent-CAT1-qfirm}
  \cos d(R_{\lambda f}x, x) \cos d(u, R_{\lambda f}x) \geq \cos d(u, x), 
\end{align} 
hold for all $x\in X$, $u\in \Argmin_X f$, and $\lambda >0$. 
\end{lemma}

We also know the following results.  
\begin{lemma}[\cite{MR3463526}*{Lemma~3.1}]
 \label{lem:KK-Delta-lsc}
 If $X$ is an admissible complete $\CAT(1)$ space, 
 then every $f\in \mathit{\Gamma}_0(X)$ 
 is $\Delta$-lower semicontinuous. 
\end{lemma}

\begin{theorem}[\cite{MR3638673}*{Theorem~4.1}]
 \label{thm:maximizer}
 Let $X$ be an admissible complete $\CAT(1)$ space, 
 $\{z_n\}$ a spherically bounded sequence in $X$, 
 $\{\beta_n\}$ a sequence of positive real numbers  
 such that $\sum_{n=1}^{\infty}\beta_n=\infty$, 
 and $g$ the real function on $X$ defined by 
 \begin{align}\label{eq:g-def}
  g(y) = \liminf_{n\to \infty} 
 \frac{1}{\sum_{l=1}^{n}\beta_l}\sum_{k=1}^{n} 
 \beta_k \cos d(y, z_k) 
 \end{align}
 for all $y\in X$. 
 Then $g$ is a $1$-Lipschitz continuous and concave function of $X$ 
 into $[0,1]$ which has a unique maximizer. 
\end{theorem}

It is obvious that if 
$A$ is a nonempty bounded subset of $\R$, 
$I$ is a closed subset of $\R$ containing $A$, 
and $f$ is a continuous and nondecreasing real function on $I$, 
then $f(\sup A) =\sup f(A)$ and $f(\inf A) =\inf f(A)$. 
Thus we obtain the following. 

\begin{lemma}\label{lem:limsup-liminf}
 Let $I$ be a nonempty closed subset of $\R$, 
 $\{t_n\}$ a bounded sequence in $I$, 
 and $f$ a continuous real function on $I$. 
 Then the following hold.   
 \begin{enumerate}
  \item[(i)] If $f$ is nondecreasing, 
    then $f(\limsup_n t_n)= \limsup_n f(t_n)$; 
  \item[(ii)] if $f$ is nonincreasing, 
    then $f(\limsup_n t_n)= \liminf_n f(t_n)$.  
 \end{enumerate}
\end{lemma}

\section{Resolvents of convex functions in $\textup{CAT}(1)$ spaces}
\label{sec:res}

In this section, we obtain three fundamental lemmas  
on the resolvents of convex functions in 
$\CAT(1)$ spaces. 

The following lemma corresponds 
to~\cite{MR2780284}*{Lemmas~3.5 and~3.6} 
in Banach spaces.   

\begin{lemma}\label{lem:demi}
 Let $X$ be an admissible complete $\CAT(1)$ space, 
 $f$ an element of $\mathit{\Gamma}_0(X)$, 
 $R_{\eta f}$ the resolvent of $\eta f$ for all $\eta >0$, 
 $\{\lambda_n\}$ a sequence of positive real numbers,  
 $p$ an element of $X$, 
 and $\{x_n\}$ a sequence in $X$. 
 Then the following hold. 
 \begin{enumerate}
  \item[(i)] If $\inf_n \lambda_n>0$, 
 $\AC\bigl(\{x_n\}\bigr)=\{p\}$, 
 and $\lim_n d(R_{\lambda_n f}x_n, x_n) = 0$, 
 then $p$ is an element of $\Argmin_X f$; 
  \item[(ii)] if $\lim_n \lambda_n = \infty$, 
 $\AC\bigl(\{R_{\lambda_n f} x_n\}\bigr)=\{p\}$, 
 and $\sup_n d(R_{\lambda_n f}x_n, x_n) <\pi/2$, 
 then $p$ is an element of $\Argmin_X f$. 
 \end{enumerate}
\end{lemma}

\begin{proof}
 Let $C_{\eta, z}$ be the real number in $\opclitvl{0}{1}$ 
 given by~\eqref{eq:C-def} 
 for all $\eta>0$ and $z\in X$.  
 It follows from~\eqref{eq:resolvent-CAT1-firm} that 
\begin{align*}
 \begin{split}
 &\bigl(
 \lambda_n C_{\lambda_n, x_n}^2 (1+C_{1, p}^2)
 + 
  C_{1, p}^2 (1+C_{\lambda_n, x_n}^2)
 \bigr)
 \cos d(R_{\lambda_n f}x_n, R_{f}p) \\
 &\geq 
 \lambda_n C_{\lambda_n, x_n}^2 (1+C_{1, p}^2)
 \cos d(R_{\lambda_n f}x_n, p)
 +
 C_{1, p}^2 (1+C_{\lambda_n, x_n}^2)
 \cos d(R_{f}p, x_n)
 \end{split}
\end{align*} 
 and hence 
\begin{align}
 \begin{split}\label{eq:lem:demi-a}
 & \cos d(R_{\lambda_n f}x_n, R_{f}p) \\
 &\geq 
  \cos d(R_{\lambda_n f}x_n, p) \\
 &\quad + \frac{C_{1,p}^2}{1+C_{1,p}^2} 
 \cdot \frac{1+C_{\lambda_n, x_n}^2}{\lambda_n C_{\lambda_n, x_n}^2}
 \bigl(\cos d(R_{f}p, x_n) - \cos d(R_{f}p, R_{\lambda_n f}x_n)\bigr)
 \end{split}
\end{align} 
 for all $n\in \N$. 

 We first show~(i). 
 Suppose that the assumptions hold. 
 Since $\lim_n C_{\lambda_n, x_n}=1$ 
 and $\inf_{n}\lambda_n>0$, 
 the sequence 
 \begin{align*}
 \ck{\frac{1+C_{\lambda_n, x_n}^2}{\lambda_n C_{\lambda_n, x_n}^2}}
 \end{align*}
 is bounded. 
 Since $t\mapsto \cos t$ is $1$-Lipschitz continuous and 
 $\lim_n d(R_{\lambda_n f}x_n, x_n) = 0$, we have 
 \begin{align*}
  \begin{split}
  \abs{\cos d(R_{f}p, x_n) 
     - \cos d(R_{f}p, R_{\lambda_n f}x_n)} 
  &\leq \abs{d(R_{f}p, x_n) 
     - d(R_{f}p, R_{\lambda_n f}x_n)} \\
  &\leq d(x_n, R_{\lambda_n f}x_n) \to 0
  \end{split}
 \end{align*}
 as $n\to \infty$. 
 Taking the lower limit in~\eqref{eq:lem:demi-a}, 
 we have 
 \begin{align*}
  \liminf_{n\to \infty} \cos d(R_{\lambda_n f}x_n, R_{f}p) 
 \geq \liminf_{n\to \infty} \cos d(R_{\lambda_n f}x_n, p). 
 \end{align*}
 It then follows from Lemma~\ref{lem:limsup-liminf} that 
 \begin{align*}
  \cos \left(\limsup_{n\to \infty} d(R_{\lambda_n f}x_n, R_{f}p)\right) 
  \geq \cos \left(\limsup_{n\to \infty} d(R_{\lambda_n f}x_n, p)\right) 
 \end{align*}
 and hence 
 \begin{align}\label{eq:lem:demi-b}
  \limsup_{n\to \infty} d(R_{\lambda_n f}x_n, R_{f}p)
   \leq 
  \limsup_{n\to \infty} d(R_{\lambda_n f}x_n, p). 
 \end{align}
 On the other hand, since $\lim_n d(R_{\lambda_n f}x_n, x_n)=0$, 
 we have 
 \begin{align}\label{eq:lem:demi-c}
  \limsup_{n\to \infty} d(R_{\lambda_n f}x_n, y) = \limsup_{n\to \infty} d(x_n, y) 
 \end{align}
 for all $y\in X$. 
 By~\eqref{eq:lem:demi-b} and~\eqref{eq:lem:demi-c}, we have 
 \begin{align*}
  \limsup_{n\to \infty} d(x_n, R_{f}p)
   \leq 
  \limsup_{n\to \infty} d(x_n, p). 
 \end{align*}
 Since $\AC\bigl(\{x_n\}\bigr)=\{p\}$, we obtain $R_{f} p = p$. 
 Consequently, it follows from~\eqref{eq:resolvent-CAT1-fix} 
 that $p \in \Argmin_X f$. 

 We next show~(ii). 
 Suppose that the assumptions hold. 
 Since 
 \begin{align*}
  \sup_n d(R_{\lambda_n f}x_n, x_n)< \frac{\pi}{2}, 
 \end{align*}
 we know that 
 \begin{align*}
  0 < \cos \sk{\sup_n d(R_{\lambda_n f}x_n, x_n)}
   = \inf_n \cos d(R_{\lambda_n f}x_n, x_n) =\inf_n C_{\lambda_n, x_n}.  
 \end{align*}
 Thus it follows from $\lim_n \lambda_n = \infty$ that 
 \begin{align*}
  0 < \frac{1+C_{\lambda_n, x_n}^2}{\lambda_n C_{\lambda_n, x_n}^2} 
 \leq \frac{2}{\lambda_n \sk{\inf_mC_{\lambda_m, x_m}}^2} \to 0 
 \end{align*}
 as $n\to \infty$. 
 Taking the lower limit in~\eqref{eq:lem:demi-a}, 
 we know that~\eqref{eq:lem:demi-b} holds. 
 Then, since $\AC\bigl(\{R_{\lambda_n f}x_n\}\bigr)=\{p\}$, 
 we obtain $R_{f}p=p$ and hence $p\in \Argmin_X f$. 
\end{proof}

We also need the following two lemmas. 

\begin{lemma}\label{lem:sqns}
 Let $X$ be an admissible complete $\CAT(1)$ space, 
 $f$ an element of $\mathit{\Gamma}_0(X)$, 
 $R_{\eta f}$ the resolvent of $\eta f$ for all $\eta >0$, 
 and 
 $\{\lambda_n\}$ a sequence of positive real numbers.  
 If $\{x_n\}$ is a sequence in $X$ such that 
 \begin{align*}
  \lim_{n\to \infty} \bigl(\cos d(u, R_{\lambda_n f}x_n) 
 - \cos d(u, x_n) \bigr) = 0 
 \quad \textrm{and} \quad 
  \sup_n d(u, x_n) < \frac{\pi}{2}
 \end{align*}
 for some $u\in \Argmin_X f$, 
 then $\lim_n d(R_{\lambda_n f}x_n, x_n)=0$. 
\end{lemma}

\begin{proof}
 It follows from~\eqref{eq:resolvent-CAT1-qfirm} that 
 \begin{align}\label{eq:lem:sqns-a}
  \cos d(u, R_{\lambda_n f}x_n) 
 \geq \cos d(R_{\lambda_n f}x_n, x_n) \cos d(u, R_{\lambda_n f}x_n) 
 \geq \cos d(u, x_n)
 \end{align}
 and hence 
 \begin{align}\label{eq:lem:sqns-b}
  \inf_n \cos d(u, R_{\lambda_n f}x_n) 
  \geq \inf_n \cos d(u, x_n)
  = \cos \sk{\sup_n d(u, x_n)} > 0
 \end{align}
 and 
 \begin{align}
  \begin{split}\label{eq:lem:sqns-c}
   0
 &\leq \cos d(u, R_{\lambda_n f}x_n) 
  \sk{1-\frac{\cos d(u, x_n)}{\cos d(u, R_{\lambda_n f}x_n)}} \\
 &= \cos d(u, R_{\lambda_n f}x_n) - \cos d(u, x_n) \to 0
  \end{split}
 \end{align}
 as $n\to \infty$. 
 Thus it follows from~\eqref{eq:lem:sqns-b} 
 and~\eqref{eq:lem:sqns-c} that 
 \begin{align}\label{eq:lem:sqns-d}
  \lim_{n\to \infty}\frac{\cos d(u, x_n)}
 {\cos d(u, R_{\lambda_n f}x_n)} = 1. 
 \end{align}
 By~\eqref{eq:lem:sqns-a} 
 and~\eqref{eq:lem:sqns-d}, we obtain 
 \begin{align*}
  1 \geq \cos d(R_{\lambda_n f}x_n, x_n) 
  \geq \frac{\cos d(u, x_n)}{\cos d(u, R_{\lambda_n f}x_n)} 
  \to 1
 \end{align*} 
 as $n\to \infty$ and hence $\lim_{n}d(R_{\lambda_n f}x_n, x_n) = 0$. 
\end{proof}

\begin{lemma}\label{lem:limsup-cos}
 Let $X$ be an admissible complete $\CAT(1)$ space, 
 $F$ a nonempty closed convex subset of $X$, 
 $P$ the metric projection of $X$ onto $F$, 
 and 
 $\{x_n\}$ a spherically bounded sequence in $X$. 
 If $\omega_{\Delta}\bigl(\{x_n\}\bigr)$ is a subset of $F$, 
 then 
 \begin{align*}
 \cos d(P v, v) \geq \limsup_{n\to \infty} \cos d(x_n, v) 
 \end{align*}
 for all $v\in X$.  
\end{lemma}

\begin{proof}
 Let $v\in X$ be given. 
 Since $\{x_n\}$ is spherically bounded, 
 Lemma~\ref{lem:EFL-subseq} 
 ensures that there exists a subsequence $\{x_{n_i}\}$ of $\{x_n\}$ 
 which is $\Delta$-convergent to some $q\in X$ and 
 \begin{align}\label{eq:lem:limsup-cos-a}
  \lim_{i\to \infty} \cos d(x_{n_i}, v) = 
  \limsup_{n\to \infty} \cos d(x_n, v). 
 \end{align}
 Since $\omega_{\Delta}\bigl(\{x_n\}\bigr) \subset F$, 
 we have $q\in F$ and hence the definition of $P$ gives us that 
 \begin{align}\label{eq:lem:limsup-cos-b}
  \cos d(P v, v) \geq \cos d(q, v). 
 \end{align}
 Since $-\cos d(\cdot, v)$ belongs to $\mathit{\Gamma}_0(X)$, 
 Lemma~\ref{lem:KK-Delta-lsc} implies that 
 it is $\Delta$-lower semicontinuous. 
 Thus we have 
  \begin{align}\label{eq:lem:limsup-cos-c}
   - \cos d(q, v) \leq 
 \liminf_{i\to \infty} \bigl(-\cos d(x_{n_i}, v) \bigr) 
= - \lim_{i\to \infty} \cos d(x_{n_i}, v).  
 \end{align}
 By~\eqref{eq:lem:limsup-cos-a},~\eqref{eq:lem:limsup-cos-b}, 
 and~\eqref{eq:lem:limsup-cos-c}, we obtain the conclusion. 
\end{proof}

\section{$\Delta$-convergent proximal-type algorithm}
\label{sec:Mann}

In this section, using 
some techniques from~\cite{MR3638673}, 
we obtain the following theorem, 
which is a generalization of Theorem~\ref{thm:KimuraKohsaka-PPA-CAT1}. 
This is the first one of our three main results in this paper. 

\begin{theorem}\label{thm:Mann-CAT1}
 Let $X$ be an admissible complete $\CAT(1)$ space, 
 $f$ an element of $\mathit{\Gamma}_0(X)$, 
 $R_{\eta f}$ the resolvent of $\eta f$ 
 for all $\eta >0$,  
 and $\{x_n\}$ a sequence defined by $x_1\in X$ 
 and~\eqref{eq:Mann-CAT1}, 
 where $\{\alpha_n\}$ is a sequence in $\clopitvl{0}{1}$ 
 and $\{\lambda_n\}$ is a sequence of positive real numbers  
 such that $\sum_{n=1}^{\infty}(1-\alpha_n) \lambda_n=\infty$.  
 Then the following hold. 
 \begin{enumerate}
  \item[(i)]  The set $\Argmin_X f$ is nonempty 
 if and only if 
 $\{R_{\lambda_n f}x_n\}$ is spherically bounded and 
 $\sup_{n} d(R_{\lambda_n f}x_n, x_n)<\pi/2$; 
  \item[(ii)] if $\Argmin_X f$ is nonempty and $\sup_n\alpha_n<1$, 
 then both $\{x_n\}$ and $\{R_{\lambda_n f}x_n\}$ 
 are $\Delta$-convergent to an element $x_{\infty}$ of $\Argmin_X f$. 
 \end{enumerate} 
\end{theorem}

\begin{proof}
 Let $C_{\eta, z}$ be the real number given by~\eqref{eq:C-def} 
 for all $\eta>0$ and $z\in X$ 
 and let $\{z_n\}$ be the sequence in $X$ 
 given by $z_n=R_{\lambda_n f}x_n$ for all $n\in \N$. 

 We first show the if part of~(i). 
 Suppose that $\{z_n\}$ is spherically bounded and 
 $\sup_{n} d(z_n, x_n)<\pi/2$ 
 and let $\{\beta_n\}$ and $\{\sigma_n\}$ be the real sequences 
 given by 
 \begin{align*}
 \beta_n = 
 \frac{(1-\alpha_n)\lambda_n C_{\lambda_n, x_n}^2}
 {1+C_{\lambda_n, x_n}^2} 
 \quad \textrm{and} \quad 
 \sigma_n = \sum_{l=1}^{n} \beta_l
 \end{align*}
 for all $n\in \N$. 
 Since $\alpha_n<1$ and $\lambda_n>0$ for all $n\in \N$ 
 and $X$ is admissible,  
 we know that $\{\beta_n\}$ is a sequence of positive real numbers.  
 Since $\sup_n d(z_n, x_n)<\pi/2$, we also know that 
 \begin{align*}
  0 < \cos \sk{\sup_{n} d(z_n, x_n)} = \inf_{n} \cos d(z_n, x_n) = \inf_{n}
  C_{\lambda_n, x_n}. 
 \end{align*}
 Thus, noting that 
 \begin{align*}
  \beta_n \geq \frac{(1-\alpha_n)\lambda_n \sk{\inf_mC_{\lambda_m, x_m}}^2}{2} 
 \end{align*} 
 and $\sum_{n=1}^{\infty}(1-\alpha_n)\lambda_n=\infty$, 
 we obtain $\sum_{n=1}^{\infty}\beta_n=\infty$. 
 According to Theorem~\ref{thm:maximizer}, 
 the real function $g$ on $X$ defined by~\eqref{eq:g-def} 
 for all $y\in X$ has a unique maximizer $p\in X$.  
 It then follows from~\eqref{eq:resolvent-CAT1-firm} that 
\begin{align*}
 \begin{split}
 &\frac{\lambda_k C_{\lambda_k, x_k}^2}{1+C_{\lambda_k, x_k}^2}
\cos d(R_{\lambda_k f}x_k, R_{f}p) \\
 &\geq 
  \frac{\lambda_k C_{\lambda_k, x_k}^2}{1+C_{\lambda_k, x_k}^2}
\cos d(R_{\lambda_k f}x_k, p) 
+ \frac{C_{1,p}^2}{1+C_{1,p}^2} 
 \bigl(\cos d(R_{f}p, x_k) - \cos d(R_{f}p, R_{\lambda_k f}x_k)\bigr)
 \end{split}
\end{align*} 
 and hence 
\begin{align}
 \begin{split}\label{eq:thm:Mann-CAT1-a}
 &\beta_k \cos d(z_k, R_{f}p) \\
 &\geq 
  \beta_k \cos d(z_k, p) + \frac{(1-\alpha_k) C_{1,p}^2}{1+C_{1,p}^2} 
 \bigl(\cos d(R_{f}p, x_k) -\cos d(R_{f}p, z_k)\bigr)
 \end{split}
\end{align} 
 for all $k\in \N$. 
 On the other hand, it follows from~\eqref{eq:CAT1-ineq} 
 and the definition of $\{x_n\}$ that 
 \begin{align}
  \begin{split}\label{eq:thm:Mann-CAT1-b}
  \cos d(R_{f}p, x_{k+1}) \geq \alpha_k \cos d(R_{f}p, x_k) 
  + (1-\alpha_k) \cos d(R_{f}p, z_k)   
  \end{split}
 \end{align}
 and hence, by~\eqref{eq:thm:Mann-CAT1-a} and~\eqref{eq:thm:Mann-CAT1-b}, 
 we have 
\begin{align}
 \begin{split}\label{eq:thm:Mann-CAT1-c}
 &\beta_k \cos d(z_k, R_{f}p) \\
 &\geq 
  \beta_k \cos d(z_k, p) + \frac{C_{1,p}^2}{1+C_{1,p}^2} 
 \bigl(\cos d(R_{f}p, x_k) -\cos d(R_{f}p, x_{k+1})\bigr)
 \end{split}
\end{align} 
 for all $k\in \N$. 
 Summing up~\eqref{eq:thm:Mann-CAT1-c} with respect to 
 $k\in \{1,2,\dots, n\}$, we obtain 
 \begin{align}
 \begin{split}\label{eq:thm:Mann-CAT1-d}
 &\frac{1}{\sigma_n}\sum_{k=1}^{n}\beta_k \cos d(z_k, R_{f}p) \\
 &\geq 
  \frac{1}{\sigma_n}\sum_{k=1}^{n}\beta_k \cos d(z_k, p) 
+ \frac{C_{1,p}^2}{1+C_{1,p}^2} 
 \cdot \frac{1}{\sigma_n} 
 \bigl(\cos d(R_{f}p, x_1) -\cos d(R_{f}p, x_{n+1})\bigr)
 \end{split}
\end{align} 
 for all $n\in \N$. 
 Since $\lim_{n}\sigma_n=\infty$, 
 it follows from~\eqref{eq:thm:Mann-CAT1-d} that 
 $g(R_{f}p) \geq g(p)$.    
 Since $p$ is the unique maximizer of $g$, 
 we obtain $R_{f}p=p$. 
 Consequently, it follows from~\eqref{eq:resolvent-CAT1-fix} 
 that $p\in \Argmin_X f$. 
 Therefore, the set $\Argmin_X f$ is nonempty. 

 We next show the only if part of~(i). 
 Suppose that $\Argmin_X f$ is nonempty and 
 let $u\in \Argmin_X f$ be given. 
 Then it follows from~\eqref{eq:resolvent-CAT1-qfirm} that 
 \begin{align}
  \begin{split}\label{eq:thm:Mann-CAT1-temp1}
   \min\bigl\{\cos d(u, z_n), \cos d(z_n, x_n)\bigr\}
   &\geq \cos d(u, z_n)\cos d(z_n, x_n) \\
   &\geq \cos d(u, x_n) 
  \end{split}
 \end{align}
 and hence 
 \begin{align}\label{eq:thm:Mann-CAT1-e}
    \max\bigl\{d(u, z_n), d(z_n,x_n)\bigr\} \leq d(u, x_n) 
 \end{align}
 for all $n\in \N$. 
 By~\eqref{eq:CAT1-ineq} and~\eqref{eq:thm:Mann-CAT1-temp1}, we have 
 \begin{align}\label{eq:thm:Mann-CAT1-f}
   \cos d(u, x_{n+1}) 
   \geq \alpha_n \cos d(u, x_{n}) + (1-\alpha_n) \cos d(u, z_{n})
   \geq \cos d(u, x_{n}). 
 \end{align}
 It then follows from~\eqref{eq:thm:Mann-CAT1-f} 
 and the admissibility of $X$ that 
 \begin{align}\label{eq:thm:Mann-CAT1-g}
  d(u, x_{n+1}) \leq d(u, x_{n}) \leq d(u, x_1) < \frac{\pi}{2}
 \end{align}
 for all $n\in \N$. 
 Thus it follows from~\eqref{eq:thm:Mann-CAT1-e} 
 and~\eqref{eq:thm:Mann-CAT1-g} that 
 \begin{align*}
  \limsup_{n\to \infty} d(u, z_{n}) \leq 
  \lim_{n\to \infty} d(u, x_{n}) \leq d(u, x_1) < \frac{\pi}{2}. 
 \end{align*}
 This implies the spherical boundedness of $\{x_n\}$ and $\{z_n\}$. 
 It also follows from~\eqref{eq:thm:Mann-CAT1-e} 
 and~\eqref{eq:thm:Mann-CAT1-g} 
 that $\sup_{n} d(z_{n}, x_{n}) < \pi/2$. 

 We finally show~(ii). 
 Suppose that $\Argmin_X f$ is nonempty and $\sup_n\alpha_n<1$. 
 Then we know
 that~\eqref{eq:thm:Mann-CAT1-temp1},~\eqref{eq:thm:Mann-CAT1-e},~\eqref{eq:thm:Mann-CAT1-f},  
 and~\eqref{eq:thm:Mann-CAT1-g} hold 
 and that both $\{x_n\}$ and $\{z_n\}$ are 
 spherically bounded. 
 Let $u\in \Argmin_X f$ be given. 
 It follows from~\eqref{eq:thm:Mann-CAT1-g} that 
 $\{d(u, x_{n})\}$ tends to some $\beta\in \clopitvl{0}{\pi/2}$.  
 By~\eqref{eq:CAT1-ineq} and~\eqref{eq:thm:Mann-CAT1-temp1}, 
 we have 
 \begin{align*}
  \begin{split}
   &\cos d(u, x_{n+1}) \\
   &\geq \alpha_n \cos d(u, x_{n}) + (1-\alpha_n) \cos d(u, z_{n}) \\
   &\geq \alpha_n \cos d(u, x_{n}) 
   + (1-\alpha_n) \cdot \frac{\cos d(u, x_{n})}{\cos d(z_{n}, x_{n})} \\
  &= \cos d(u, x_{n}) 
   + (1-\alpha_n) \cos d(u, x_{n}) 
   \left(\frac{1}{\cos d(z_{n},x_{n})} -1 \right)  
  \end{split}
 \end{align*}
 and hence 
 \begin{align}\label{eq:thm:Mann-CAT1-h}
  0 \leq (1-\alpha_n) \left(\frac{1}{\cos d(z_{n},x_{n})} -1 \right)
      \leq \frac{\cos d(u, x_{n+1})}{\cos d(u, x_{n})} - 1 \to 
  \frac{\cos \beta}{\cos \beta} - 1 = 0
 \end{align}
 as $n\to \infty$. 
 Since $\sup_n \alpha_n <1$, 
 it follows from~\eqref{eq:thm:Mann-CAT1-h} that 
 \begin{align}\label{eq:thm:Mann-CAT1-i}
  \lim_{n\to \infty} d(z_{n}, x_{n}) = 0. 
 \end{align} 

 On the other hand, 
 it follows from~\eqref{eq:resolvent-CAT1-sqfirm} 
 and~\eqref{eq:thm:Mann-CAT1-i} that 
 there exists a positive real number $K$ such that  
 \begin{align}\label{eq:thm:Mann-CAT1-j}
  \lambda_n \bigl(f(z_n) -f(u)\bigr) 
  \leq \frac{K\pi}{2} 
  \bigl(\cos d(u, z_n) - \cos d(u, x_n)\bigr) 
 \end{align} 
 for all $n\in \N$. 
 It then follows from~\eqref{eq:thm:Mann-CAT1-f} 
 and~\eqref{eq:thm:Mann-CAT1-j} that 
 \begin{align*}
  (1-\alpha_n)\lambda_n \bigl(f(z_n) -f(u)\bigr) 
  \leq \frac{K\pi}{2} 
  \bigl(\cos d(u, x_{n+1}) - \cos d(u, x_n)\bigr) 
 \end{align*} 
 and hence 
 \begin{align}\label{eq:thm:Mann-CAT1-k}
  \sum_{n=1}^{\infty} (1-\alpha_n)\lambda_n \bigl(f(z_n) -f(u)\bigr)
  \leq \frac{K\pi}{2} \bigl(\cos \beta - \cos d(u, x_1)\bigr) 
  <\infty. 
 \end{align}
 Since $\sum_{n=1}^{\infty}(1-\alpha_n)\lambda_n=\infty$, 
 it follows from~\eqref{eq:thm:Mann-CAT1-k} that 
 \begin{align}\label{eq:thm:Mann-CAT1-l}
  \liminf_{n\to \infty} \bigl(f(z_n) -f(u)\bigr) =0. 
 \end{align}
 By the definitions of $\{x_n\}$ and $\{z_n\}$ 
 and the convexity of $f$, we also have 
 \begin{align*}
  -\infty < \inf f(X)\leq f(z_n) 
 \leq f(z_n) + 
 \frac{1}{\lambda_n}\tan d(z_n, x_n) \sin d(z_n, x_n) 
 \leq f(x_n)   
 \end{align*}
 and 
 \begin{align*}
  -\infty < \inf f(X) \leq f(x_{n+1}) 
  \leq \alpha_n f(x_n) + (1-\alpha_n) f(z_n) 
  \leq f(x_n)
 \end{align*}
 for all $n\in \N$. 
 Thus $\{f(x_n)\}$ tends to some $\gamma\in \R$ and 
 $\{f(z_n)\}$ is bounded. 
 Let $\{n_i\}$ be any increasing sequence in $\N$. 
 Since $\sup_n \alpha_n<1$, 
 we have a subsequence $\{n_{i_j}\}$ of $\{n_i\}$ 
 such that $\{\alpha_{n_{i_j}}\}$ tends to 
 some $\delta\in \clopitvl{0}{1}$. 
 Then, letting $j\to \infty$ in 
 \begin{align*}
  \frac{1}{1-\alpha_{n_{i_j}}} \Bigl(f\bigl(x_{n_{i_j}+1}\bigr) 
  - \alpha_{n_{i_j}} f\bigl(x_{n_{i_j}}\bigr) \Bigr) 
  \leq f\bigl(z_{n_{i_j}}\bigr) \leq f\bigl(x_{n_{i_j}}\bigr), 
 \end{align*}
 we obtain $f\bigl(z_{n_{i_j}}\bigr)\to \gamma$. 
 Thus $\{f(z_n)\}$ also tends to $\gamma$. 
 Consequently, it follows from~\eqref{eq:thm:Mann-CAT1-l} 
 that  
 \begin{align}\label{eq:thm:Mann-CAT1-m} 
  \lim_{n\to \infty} f(x_n) = \gamma = f(u) = \inf f(X). 
 \end{align}

 Let $z$ be any element of $\omega_{\Delta}\bigl(\{x_n\}\bigr)$. 
 Then we have a subsequence $\{x_{m_i}\}$ of $\{x_n\}$ 
 which is $\Delta$-convergent to $z$. 
 Since $f$ is $\Delta$-lower semicontinuous 
 by Lemma~\ref{lem:KK-Delta-lsc}, 
 it follows from~\eqref{eq:thm:Mann-CAT1-m} that 
 \begin{align*}
  f(z) \leq \liminf_{i\to \infty} f(x_{m_i}) 
  = \lim_{n\to \infty}f(x_n) = \inf f(X) 
 \end{align*}
 and hence $z \in \Argmin_X f$. 
 Thus $\omega_{\Delta}\bigl(\{x_n\}\bigr)$ 
 is a subset of $\Argmin_X f$. 
 It then follows from~\eqref{eq:thm:Mann-CAT1-g} 
 that $\{d(z, x_n)\}$ is convergent for each 
 $z\in \omega_{\Delta}\bigl(\{x_n\}\bigr)$. 
 Thus, Lemma~\ref{lem:KSY-conv} ensures that 
 $\{x_n\}$ is $\Delta$-convergent to some $x_{\infty}\in X$. 
 Since 
 \begin{align*}
  \{x_{\infty}\} = \omega_{\Delta} \bigl(\{x_n\}\bigr) \subset
  \Argmin_{X} f, 
 \end{align*}
 we know that $x_{\infty}$ belongs to $\Argmin_X f$. 
 It then follows from~\eqref{eq:thm:Mann-CAT1-i} that 
 \begin{align*}
  \AC\bigl(\{z_{l_i}\}\bigr) = \AC\bigl(\{x_{l_i}\}\bigr) = \{x_{\infty}\}
 \end{align*} 
 for each increasing sequence $\{l_i\}$ in $\N$. 
 Consequently, we conclude that both $\{x_n\}$ and $\{z_n\}$ 
 are $\Delta$-convergent to an element $x_{\infty}$ 
 of $\Argmin_X f$. 
\end{proof}

As a direct consequence of Theorem~\ref{thm:Mann-CAT1}, 
we obtain the following corollary. 

\begin{corollary}\label{cor:Mann-CAT1-S_H}
 Let $(S_H, \rho_{S_H})$ be a Hilbert sphere, 
 $X$ an admissible closed convex subset of $S_H$, 
 $f$ an element of $\mathit{\Gamma}_0(X)$, 
 and $\{x_n\}$ a sequence defined 
 by $x_1\in X$ and~\eqref{eq:Mann-CAT1}, where 
 $\{\alpha_n\}$ is a sequence in $\clopitvl{0}{1}$ 
 and $\{\lambda_n\}$ is a sequence of positive real numbers 
 such that $\sum_{n=1}^{\infty}(1-\alpha_n) \lambda_n=\infty$. 
 Then $\Argmin_X f$ is nonempty 
 if and only if 
 $\{R_{\lambda_n f}x_n\}$ is spherically bounded and 
 $\sup_{n} \rho_{S_H}(R_{\lambda_n f}x_n, x_n)<\pi/2$. 
 Further, if $\Argmin_X f$ is nonempty and 
 $\sup_n\alpha_n<1$, 
  then both $\{x_n\}$ and $\{R_{\lambda_n f}x_n\}$ 
 are $\Delta$-convergent to an element $x_{\infty}$ of $\Argmin_X f$. 
\end{corollary}

\section{Convergent proximal-type algorithm}
\label{sec:Halpern}

In this section, 
using some techniques from~\cite{MR3020188}*{Theorem~3.2}, 
we first obtain the following theorem. 
This is the second one of our three main results in this paper. 

\begin{theorem}\label{thm:Halpern-CAT1}
 Let $X$ be an admissible complete $\CAT(1)$ space, 
 $f$ an element of $\mathit{\Gamma}_0(X)$, 
 $R_{\eta f}$ the resolvent of $\eta f$ 
 for all $\eta >0$, 
 $v$ an element of $X$,  
 and $\{y_n\}$ a sequence defined by $y_1\in X$ 
 and~\eqref{eq:Halpern-CAT1}, 
 where $\{\alpha_n\}$ is a sequence in $[0,1]$ and 
 $\{\lambda_n\}$ is a sequence of positive real numbers 
 such that $\lim_{n} \lambda_n=\infty$.  
 Then the following hold. 
 \begin{enumerate}
  \item[(i)] The set $\Argmin_X f$ is nonempty 
 if and only if 
 $\{R_{\lambda_n f}y_n\}$ is spherically bounded and 
 $\sup_{n} d(R_{\lambda_n f}y_n, y_n)<\pi/2$; 
  \item[(ii)] if $\Argmin_X f$ is nonempty, $\lim_{n} \alpha_n = 0$, 
 and $\sum_{n=1}^{\infty}\alpha_n^2=\infty$, 
  then both $\{y_n\}$ and $\{R_{\lambda_n f}y_n\}$ 
 are convergent to $Pv$, 
 where $P$ denotes the metric projection 
 of $X$ onto $\Argmin_X f$. 
 \end{enumerate}
\end{theorem}

\begin{proof}
 Let $\{z_n\}$ be the sequence in $X$ given by 
 $z_n = R_{\lambda_n f}y_n$ for all $n\in \N$. 

 We first show the if part of~(i). 
 Suppose that $\{z_n\}$ is spherically bounded and 
 $\sup_{n} d(z_n, y_n)<\pi/2$. 
 Then Lemma~\ref{lem:EFL-subseq} implies that 
 there exists $p\in X$ such that $\AC\bigl(\{z_n\}\bigr)=\{p\}$. 
 Since 
 \begin{align*}
  \lim_{n\to \infty}\lambda_n=\infty 
 \quad \textrm{and} \quad 
  \sup_{n} d(z_n, y_n)<\frac{\pi}{2}, 
 \end{align*}
 it follows from~(ii) of Lemma~\ref{lem:demi} 
 that $p\in \Argmin_X f$. 
 Thus $\Argmin_X f$ is nonempty. 

 We next show the only if part of~(i). 
 Suppose that $\Argmin_X f$ is nonempty 
 and let $P$ be the metric projection of $X$ onto $\Argmin_X f$.  
 It follows from~\eqref{eq:resolvent-CAT1-qfirm} that 
 \begin{align}\label{eq:thm:Halpern-CAT1-a-b}
    \max\bigl\{d(Pv, z_n), d(z_n, y_n)\bigr\} \leq d(Pv, y_n) 
 \end{align}
 for all $n\in \N$.  
 It also follows from~\eqref{eq:CAT1-ineq} 
 and~\eqref{eq:resolvent-CAT1-qfirm} that 
 \begin{align*}
   \cos d(Pv, y_{n+1}) 
   &\geq \alpha_n \cos d(Pv, v) + (1-\alpha_n) \cos d(Pv, y_{n}). 
 \end{align*}
 This implies that 
 \begin{align*}
  \cos d(Pv, y_{n}) \geq \min\bigl\{\cos d(Pv, v), \cos d(Pv, y_1)\bigr\}
 \end{align*}
 for all $n\in \N$ and hence 
 \begin{align}\label{eq:thm:Halpern-CAT1-c}
  d(Pv, y_{n}) \leq \max\bigl\{d(Pv, v), d(Pv, y_1)\bigr\} < \frac{\pi}{2}
 \end{align}
 for all $n\in \N$, where the last inequality 
 follows from the admissibility of $X$. 
 Then, by~\eqref{eq:thm:Halpern-CAT1-a-b} 
 and~\eqref{eq:thm:Halpern-CAT1-c}, 
 we see that both $\{y_n\}$ and $\{z_n\}$ 
 are spherically bounded and 
 \begin{align}\label{eq:thm:Halpern-CAT1-d}
  \sup_{n} d(z_n, y_n) < \frac{\pi}{2}. 
 \end{align}

 We finally show~(ii). 
 Suppose that $\Argmin_X f$ is nonempty, $\lim_n \alpha_n = 0$, 
 and $\sum_{n=1}^{\infty}\alpha_n^2=\infty$. 
 Then we know 
 that~\eqref{eq:thm:Halpern-CAT1-a-b},~\eqref{eq:thm:Halpern-CAT1-c},
 and~\eqref{eq:thm:Halpern-CAT1-d} 
 hold and that both $\{y_n\}$ and $\{z_n\}$ are spherically bounded. 
 By~\eqref{eq:KS-ineq-Halpern} 
 and~\eqref{eq:thm:Halpern-CAT1-a-b}, 
 we have 
 \begin{align}
 \begin{split}\label{eq:thm:Halpern-CAT1-e}
 &\cos d\bigl(Pv, y_{n+1}\bigr) \\
 &\geq (1-\beta_n) \cos d(Pv, y_n) 
 + \beta_n \cdot 
 \frac{\cos d(Pv, v)}{\sin d(z_n, v) 
 \tan \bigl(\frac{\alpha_n}{2}d(z_n, v)\bigr) 
 + \cos d(z_n, v)}, 
 \end{split}
\end{align}
where 
\begin{align}\label{eq:thm:Halpern-CAT1-beta_n}
 \beta_n
  =
 \begin{cases}
   1 - \dfrac{\sin \bigl((1-\alpha_n) d(z_n, v)\bigr)}{\sin d(z_n, v)} 
   &(z_n \neq v); \\
   \alpha_n & (z_n = v)
  \end{cases}
 \end{align}
 for all $n\in \N$. 
 Note that if $z_n\neq v$, then it follows from 
 \begin{align*}
  \begin{split}
  \sin \bigl((1-\alpha_n)d(z_n, v)\bigr) 
  \geq \alpha_n \sin 0 +(1-\alpha_n) \sin d(z_n, v) 
  = (1-\alpha_n) \sin d(z_n, v) 
  \end{split}  
 \end{align*}
 that $\alpha_n\geq \beta_n$. 
 Hence we have 
 \begin{align}\label{eq:thm:Halpern-CAT1-alpha-beta-a}
  \alpha_n\geq \beta_n
 \end{align}
 for all $n\in \N$. 
 Letting 
 \begin{align}\label{eq:thm:Halpern-CAT1-s_n}
  s_n = 1 - \cos d(Pv, y_n) 
 \end{align}
 and 
 \begin{align}\label{eq:thm:Halpern-CAT1-t_n}
 t_n = 1 - 
   \frac{\cos d(Pv, v)}{\sin d(z_n, v) 
 \tan \bigl(\frac{\alpha_n}{2}d(z_n, v)\bigr) 
 + \cos d(z_n, v)}, 
 \end{align}
 we have from~\eqref{eq:thm:Halpern-CAT1-e} that 
 \begin{align}\label{eq:thm:Halpern-CAT1-f}
  s_{n+1} \leq (1-\beta_n) s_n + \beta_n t_n
 \end{align}
 for all $n\in \N$. 

 Let us show that $\sum_{n=1}^{\infty}\beta_n=\infty$. 
 Suppose that $z_n \neq v$. 
 If $\gamma \in [0,1]$ and 
 $\varphi (t) = \sin (\gamma t)/ \sin t$ 
 for all $t \in \opclitvl{0}{\pi/2}$, then we have 
 \begin{align*}
 \begin{split}
 \varphi'(t) 
 &=\frac{\cos (\gamma t) \cos t}{\sin ^2 t} 
 \bigl(\gamma \tan t - \tan (\gamma t)\bigr) \\
 &\geq \frac{\cos (\gamma t) \cos t}{\sin ^2 t} 
 \Bigl(\gamma \tan t - \bigl((1-\gamma )\tan 0 +\gamma\tan t\bigr)
 \Bigr) =0
 \end{split} 
 \end{align*}
 for all $t\in \openitvl{0}{\pi/2}$. 
 Thus the function 
 \begin{align*}
  t\mapsto \frac{\sin \bigl((1-\alpha_n)t\bigr)}{\sin t}
 \end{align*}
 is nondecreasing on $\opclitvl{0}{\pi/2}$. 
 Using this property and 
 the inequality $1-t^2/4 \geq \cos t$ on $[0,\pi/2]$, 
 we have 
 \begin{align*}
 \beta_n 
 \geq 1 - \sin \frac{(1-\alpha_n)\pi}{2} 
 =1-\cos \frac{\alpha_n \pi}{2} 
 \geq \frac{1}{4} \sk{\frac{\alpha_n \pi}{2}}^2 
 = \frac{\alpha_n^2 \pi^2}{16}. 
 \end{align*}
 If $z_n=v$, then 
 $\beta_n =\alpha_n \geq \alpha_n^2 \geq 16^{-1}\pi ^2\alpha_n^2$.  
 Hence the inequality 
 \begin{align}\label{eq:thm:Halpern-CAT1-alpha-beta-b}
  \beta_n \geq \frac{\pi ^2}{16}\alpha_n^2
 \end{align} 
 holds for all $n\in \N$. 
 It then follows from $\sum_{n=1}^{\infty}\alpha_n^2=\infty$ 
 that $\sum_{n=1}^{\infty}\beta_n=\infty$.  

 We next show that $\limsup_{n} t_n \leq 0$. 
 Since $\lim_{n}\alpha_n=0$, we have 
 \begin{align*}
  \lim_{n\to \infty} \sin d(z_n, v)\tan \frac{\alpha_n d(z_n, v)}{2} = 0. 
 \end{align*}
 If $\limsup_{n} \cos d(z_n, v) = 0$, then we have 
 \begin{align*}
  \limsup_{n\to \infty} t_n 
 = 1 - 
   \liminf_{n\to \infty}\frac{\cos d(Pv, v)}{\sin d(z_n, v) 
 \tan \bigl(\frac{\alpha_n}{2}d(z_n, v)\bigr) 
 + \cos d(z_n, v)} 
 = -\infty. 
 \end{align*}
If $\limsup_{n} \cos d(z_n, v) > 0$,  then 
it follows from Lemma~\ref{lem:limsup-liminf} that 
 \begin{align}
  \begin{split}\label{eq:thm:Halpern-CAT1-g}
 & \limsup_{n\to \infty} t_n \\
 &= 1 - 
   \liminf_{n\to \infty}\frac{\cos d(Pv, v)}{\sin d(z_n, v) 
 \tan \bigl(\frac{\alpha_n}{2}d(z_n, v)\bigr) 
 + \cos d(z_n, v)} \\   
 &= 1 - 
  \frac{\cos d(Pv, v)}{
 \limsup_{n\to \infty}\Bigl(\sin d(z_n, v) 
 \tan \bigl(\frac{\alpha_n}{2}d(z_n, v)\bigr) 
 + \cos d(z_n, v)\Bigr)} \\   
 &= 1 - 
  \frac{\cos d(Pv, v)}{\limsup_{n\to \infty} \cos d(z_n, v)}. 
  \end{split}
 \end{align}
 On the other hand, if $q$ is any element of
 $\omega_{\Delta}\bigl(\{z_n\}\bigr)$, 
 then there exists a subsequence $\{z_{n_i}\}$ of $\{z_n\}$ 
 which is $\Delta$-convergent to some $q\in X$. 
 Since 
 \begin{align*}
  \lim_{i\to \infty}\lambda_{n_i}=\infty, 
  \quad 
  \AC\bigl(\{R_{\lambda_{n_i}f}y_{n_i}\}\bigr)=\{q\}, 
  \quad \textrm{and} \quad 
  \sup_{i} d(R_{\lambda_{n_i}f}y_{n_i}, y_{n_i}) < \frac{\pi}{2}, 
 \end{align*}
 it follows from~(ii) of Lemma~\ref{lem:demi} that 
 $q\in \Argmin_X f$. 
 Thus $\omega_{\Delta}\bigl(\{z_n\}\bigr)$ 
 is a subset of $\Argmin_X f$. 
 It then follows from Lemma~\ref{lem:limsup-cos} that 
 \begin{align}\label{eq:thm:Halpern-CAT1-j}
  \cos d(Pv, v) \geq \limsup_{n\to \infty} \cos d(z_{n}, v). 
 \end{align}
 By~\eqref{eq:thm:Halpern-CAT1-g} 
 and~\eqref{eq:thm:Halpern-CAT1-j}, 
 we know that $\limsup_{n}t_n\leq 0$. 

 Therefore, Lemma~\ref{lem:AKTT-seq} yields that 
 $\lim_{n} s_n =0$ and hence 
 \begin{align}\label{eq:thm:Halpern-CAT1-k}
  \lim_{n\to \infty} d(Pv, y_n) = 0.  
 \end{align}
 Consequently, by~\eqref{eq:thm:Halpern-CAT1-a-b} 
 and~\eqref{eq:thm:Halpern-CAT1-k}, 
 we conclude that both $\{y_n\}$ and $\{z_n\}$ are 
 convergent to $Pv$. 
\end{proof}

Using Lemma~\ref{lem:KSY-seq}, 
we next obtain the last one of our three 
main results in this paper. 

\begin{theorem}\label{thm:Halpern-CAT1-another}
 Let $X$, $f$, $\{R_{\eta f}\}_{\eta>0}$, and $v$ 
 be the same as in Theorem~\ref{thm:Halpern-CAT1}, 
 and $\{y_n\}$ a sequence defined by $y_1\in X$ 
 and~\eqref{eq:Halpern-CAT1}, 
 where $\{\alpha_n\}$ is a sequence in $\opclitvl{0}{1}$ 
 and $\{\lambda_n\}$ is a sequence of positive real numbers  
 such that
 \begin{align*}
 \lim_{n\to \infty} \alpha_n = 0, \quad  
 \sum_{n=1}^{\infty}\alpha_n^2=\infty, 
 \quad \textrm{and} \quad \inf_n \lambda_n > 0.  
 \end{align*}
  If $\Argmin_X f$ is nonempty, 
 then both $\{y_n\}$ and $\{R_{\lambda_n f}y_n\}$ 
 are convergent to $Pv$, 
 where $P$ denotes the metric projection 
 of $X$ onto $\Argmin_X f$. 
\end{theorem}

\begin{proof}
 Let $\{z_n\}$ be the sequence in $X$ 
 given by $z_n=R_{\lambda_n f}y_n$ for all $n\in \N$. 

 As in the proof of Theorem~\ref{thm:Halpern-CAT1}, 
 we can see  
 that~\eqref{eq:thm:Halpern-CAT1-a-b},~\eqref{eq:thm:Halpern-CAT1-c},
 and~\eqref{eq:thm:Halpern-CAT1-d} hold 
 and that both $\{y_n\}$ and $\{z_n\}$ are spherically bounded. 
 Let $\{\beta_n\}$, $\{s_n\}$, and $\{t_n\}$ 
 be the real sequences defined 
 by~\eqref{eq:thm:Halpern-CAT1-beta_n},~\eqref{eq:thm:Halpern-CAT1-s_n},
 and~\eqref{eq:thm:Halpern-CAT1-t_n}, respectively. 
 Then we can see
 that~\eqref{eq:thm:Halpern-CAT1-alpha-beta-a},~\eqref{eq:thm:Halpern-CAT1-f}, 
 and~\eqref{eq:thm:Halpern-CAT1-alpha-beta-b} hold 
 for all $n\in \N$. 
 Since $\{\alpha_n\}$ is a sequence in $\opclitvl{0}{1}$, so is $\{\beta_n\}$. 
 Since $\sum_{n=1}^{\infty}\alpha_n^2 =\infty$, 
 it follows from~\eqref{eq:thm:Halpern-CAT1-alpha-beta-b} 
 that $\sum_{n=1}^{\infty}\beta_n=\infty$. 

 Let $\{n_i\}$ be any increasing sequence in $\N$ such that 
 \begin{align}\label{eq:thm:Halpern-CAT1-another-a0}
  \limsup_{i\to \infty} \bigl(s_{n_i}-s_{n_i+1}\bigr) \leq 0. 
 \end{align}
 Then we show that 
 \begin{align}\label{eq:thm:Halpern-CAT1-another-key}
  \limsup_{i\to \infty}t_{n_i}\leq 0. 
 \end{align}
 If $\limsup_{i} \cos d(z_{n_i}, v) = 0$, then we have 
 \begin{align*}
  \limsup_{i\to \infty} t_{n_i} 
 = 1 - 
   \liminf_{i\to \infty}\frac{\cos d(Pv, v)}{\sin d(z_{n_i}, v) 
 \tan \bigl(\frac{\alpha_{n_i}}{2}d(z_{n_i}, v)\bigr) 
 + \cos d(z_{n_i}, v)} = -\infty. 
 \end{align*}
 If $\limsup_{i} \cos d(z_{n_i}, v) > 0$,  
 then it follows from Lemma~\ref{lem:limsup-liminf} that 
 \begin{align}\label{eq:thm:Halpern-CAT1-another-a}
  \limsup_{i\to \infty} t_{n_i} = 1 - 
  \frac{\cos d(Pv, v)}{\limsup_{i\to \infty} \cos d(z_{n_i}, v)}. 
 \end{align}
 It follows from~\eqref{eq:CAT1-ineq} that 
 \begin{align*}
  \begin{split}
 s_{n_i}-s_{n_i+1} 
 &=\cos d(Pv, y_{{n_i}+1}) - \cos d(Pv, y_{n_i}) \\
 &\geq \alpha_{n_i} \cos d(Pv, v) + (1-\alpha_{n_i})\cos d(Pv, z_{n_i})
 - \cos d(Pv, y_{n_i}). 
  \end{split}
 \end{align*}
 Hence~\eqref{eq:resolvent-CAT1-qfirm} yields that 
 \begin{align}
  \begin{split}\label{eq:thm:Halpern-CAT1-another-s-s}
 &s_{n_i}-s_{n_i+1} 
 + \alpha_{n_i} \bigl(\cos d(Pv, z_{n_i}) - \cos d(Pv, v)\bigr)\\
 &\quad \geq \cos d(Pv, z_{n_i}) - \cos d(Pv, y_{n_i}) \\
 &\quad \geq 0. 
  \end{split}
 \end{align}
 Since $\lim_{i}\alpha_{n_i}=0$, 
 it follows from~\eqref{eq:thm:Halpern-CAT1-another-a0} 
 and~\eqref{eq:thm:Halpern-CAT1-another-s-s} that 
 \begin{align*}
  \begin{split}
  \lim_{i\to \infty} \bigl(
  \cos d(Pv, z_{n_i}) - \cos d(Pv, y_{n_i}) \bigr) = 0. 
  \end{split}
 \end{align*}
 On the other hand, it follows from~\eqref{eq:thm:Halpern-CAT1-c} 
 that 
 \begin{align*}
   \sup_{i} d(Pv, y_{n_i}) \leq \sup_{n} d(Pv, y_{n}) < \frac{\pi}{2}. 
 \end{align*}
 Thus it follows from Lemma~\ref{lem:sqns} that 
 $\lim_{i} d(z_{n_i}, y_{n_i})=0$. 
 Let $q$ be any element of $\omega_{\Delta}\bigl(\{z_{n_i}\}\bigr)$. 
 Then there exists a subsequence $\{z_{n_{i_j}}\}$ of $\{z_{n_i}\}$ 
 which is $\Delta$-convergent to some $q\in X$. 
 Since 
 \begin{align*}
  \inf_{j}\lambda_{n_{i_j}} > 0, 
  \quad 
  \AC\bigl(\{y_{n_{i_j}}\}\bigr)=\{q\}, 
  \quad \textrm{and} \quad 
  \lim_{j\to \infty} d\bigl(R_{\lambda_{n_{i_j}}f}y_{n_{i_j}}, y_{n_{i_j}}\bigr) = 0, 
 \end{align*}
 it follows from~(i) of Lemma~\ref{lem:demi} that 
 $q\in \Argmin_X f$. 
 Thus $\omega_{\Delta}\bigl(\{z_{n_i}\}\bigr)$ 
 is a subset of $\Argmin_X f$. 
 Then, by Lemma~\ref{lem:limsup-cos}, 
 we know that  
 \begin{align}\label{eq:thm:Halpern-CAT1-another-c}
  \cos d(Pv, v) \geq \limsup_{i\to \infty} \cos d(z_{n_i}, v). 
 \end{align}
 By~\eqref{eq:thm:Halpern-CAT1-another-a}
 and~\eqref{eq:thm:Halpern-CAT1-another-c}, 
 we know that~\eqref{eq:thm:Halpern-CAT1-another-key} holds. 

 Therefore, Lemma~\ref{lem:KSY-seq} yields that 
 $\lim_{n} s_n =0$ and hence 
 \begin{align}\label{eq:thm:Halpern-CAT1-another-d}
   \lim_{n\to \infty} d(Pv, y_n) = 0.  
 \end{align}
 Consequently, by~\eqref{eq:thm:Halpern-CAT1-a-b} 
 and~\eqref{eq:thm:Halpern-CAT1-another-d}, 
 we conclude that both $\{y_n\}$ and $\{z_n\}$ are 
 convergent to $Pv$. 
\end{proof}

As direct consequences of 
Theorems~\ref{thm:Halpern-CAT1} 
and~\ref{thm:Halpern-CAT1-another}, 
we obtain the following two corollaries, respectively.  

\begin{corollary}\label{cor:Halpern-CAT1-S_H}
 Let $(S_H, \rho_{S_H})$ be a Hilbert sphere, 
 $X$ an admissible closed convex subset of $S_H$, 
 $f$ an element of $\mathit{\Gamma}_0(X)$, 
 $v$ an element of $X$,  
 and $\{y_n\}$ a sequence defined by $y_1\in X$ 
 and~\eqref{eq:Halpern-CAT1}, 
 where $\{\alpha_n\}$ is a sequence in $[0,1]$ 
 and $\{\lambda_n\}$ is a sequence of positive real numbers  
 such that $\lim_{n} \lambda_n=\infty$. 
 Then $\Argmin_X f$ is nonempty 
 if and only if 
 $\{R_{\lambda_n f}y_n\}$ is spherically bounded and 
 $\sup_{n} \rho_{S_H}(R_{\lambda_n f}y_n, y_n)<\pi/2$. 
 Further, if $\Argmin_X f$ is nonempty, $\lim_{n} \alpha_n = 0$, 
 and $\sum_{n=1}^{\infty}\alpha_n^2=\infty$, 
 then both $\{y_n\}$ and $\{R_{\lambda_n f}y_n\}$ 
 are convergent to $Pv$, 
 where $P$ denotes the metric projection 
 of $X$ onto $\Argmin_X f$. 
\end{corollary}

\begin{corollary}\label{cor:Halpern-CAT1-S_H-another}
 Let $(S_H, \rho_{S_H})$, $X$, $f$, and $v$ 
 be the same as in Corollary~\ref{cor:Halpern-CAT1-S_H} 
 and $\{y_n\}$ a sequence defined by $y_1\in X$ 
 and~\eqref{eq:Halpern-CAT1}, 
 where $\{\alpha_n\}$ is a sequence in $\opclitvl{0}{1}$ 
 and $\{\lambda_n\}$ is a sequence of positive real numbers  
 such that $\lim_n \alpha_n = 0$, 
 $\sum_{n=1}^{\infty}\alpha_n^2=\infty$, 
 and $\inf_{n} \lambda_n >0$. 
 If $\Argmin_X f$ is nonempty, 
 then both $\{y_n\}$ and $\{R_{\lambda_n f}y_n\}$ 
 are convergent to $Pv$, 
 where $P$ denotes the metric projection 
 of $X$ onto $\Argmin_X f$. 
\end{corollary}

\section{Results in $\textup{CAT}(\kappa)$ spaces with a positive $\kappa$}
\label{sec:cor}

In this final section, 
using Theorems~\ref{thm:Mann-CAT1},~\ref{thm:Halpern-CAT1}, 
and~\ref{thm:Halpern-CAT1-another}, we deduce three corollaries 
in $\CAT(\kappa)$ spaces with a positive real number $\kappa$. 

Throughout this section, we suppose the following. 
\begin{itemize}
 \item $\kappa$ is a positive real number and
       $D_{\kappa}=\pi/\sqrt{\kappa}$; 
 \item $X$ is a complete $\CAT(\kappa)$ space such that 
$d(w,w')<D_{\kappa}/2$ for all $w,w'\in X$; 
 \item $f$ is a proper lower semicontinuous convex function 
of $X$ into $\opclitvl{-\infty}{\infty}$; 
 \item $\tilde{R}_{\eta f}$ is the mapping of $X$ into itself defined by 
 \begin{align*}
  \tilde{R}_{\eta f} x = \Argmin_{y\in X} 
 \left\{f(y) + \frac{1}{\eta}\tan \bigl(\sqrt{\kappa} d(y, x)\bigr) 
 \sin \bigl(\sqrt{\kappa} d(y, x)\bigr)\right\}
 \end{align*}
 for all $\eta>0$ and $x\in X$. 
\end{itemize}

Since the space $(X, \sqrt{\kappa}d)$ is an admissible complete 
$\CAT(1)$ space, 
the mapping $\tilde{R}_{\eta f}$ is well defined 
and Theorems~\ref{thm:Mann-CAT1},~\ref{thm:Halpern-CAT1}, 
and~\ref{thm:Halpern-CAT1-another} immediately imply 
the following three corollaries, respectively. 

\begin{corollary}\label{cor:Mann-CATk}
 Let $\{x_n\}$ be a sequence defined by $x_1\in X$ and 
 \begin{align*}
  x_{n+1} = \alpha_n x_n \oplus (1-\alpha_n) \tilde{R}_{\lambda_n f}x_n 
 \quad (n=1,2,\dots), 
 \end{align*}
 where $\{\alpha_n\}$ is a sequence in $\clopitvl{0}{1}$ and 
 $\{\lambda_n\}$ is a sequence of positive real numbers  
 such that $\sum_{n=1}^{\infty}(1-\alpha_n) \lambda_n=\infty$. 
 Then $\Argmin_X f$ is nonempty 
 if and only if 
 \begin{align*}
 \inf_{y\in X} \limsup_{n\to \infty}
 d(y, \tilde{R}_{\lambda_n f}x_n) < \frac{D_{\kappa}}{2} 
 \quad \textrm{and} \quad 
 \sup_{n} d(\tilde{R}_{\lambda_n f}x_n, x_n)<
  \frac{D_{\kappa}}{2}. 
 \end{align*}
 Further, if $\Argmin_X f$ is nonempty and $\sup_n\alpha_n<1$, 
 then both $\{x_n\}$ and $\{\tilde{R}_{\lambda_n f}x_n\}$ 
 are $\Delta$-convergent to an element $x_{\infty}$ of $\Argmin_X f$. 
\end{corollary}

\begin{corollary}\label{cor:Halpern-CATk}
 Let $v$ be an element of $X$ 
 and $\{y_n\}$ a sequence defined by $y_1\in X$ and 
 \begin{align}\label{eq:Halpern-CATk}
  y_{n+1} = \alpha_n v \oplus (1-\alpha_n) \tilde{R}_{\lambda_n f}y_n 
 \quad (n=1,2,\dots), 
 \end{align}
 where $\{\alpha_n\}$ is a sequence in $[0,1]$ and 
 $\{\lambda_n\}$ is a sequence of positive real numbers  
 such that $\lim_{n} \lambda_n=\infty$. 
 Then $\Argmin_X f$ is nonempty 
 if and only if 
 \begin{align*}
 \inf_{y\in X} \limsup_{n\to \infty}
 d(y, \tilde{R}_{\lambda_n f}y_n) < \frac{D_{\kappa}}{2} 
 \quad \textrm{and} \quad 
 \sup_{n} d(\tilde{R}_{\lambda_n f}y_n, y_n)<
  \frac{D_{\kappa}}{2}. 
 \end{align*}
 Further, if $\Argmin_X f$ is nonempty, $\lim_n \alpha_n = 0$, 
 and $\sum_{n=1}^{\infty}\alpha_n^2=\infty$, 
 then both $\{y_n\}$ and $\{\tilde{R}_{\lambda_n f}y_n\}$ 
 are convergent to $Pv$, 
 where $P$ denotes the metric projection 
 of $X$ onto $\Argmin_X f$. 
\end{corollary}

\begin{corollary}\label{cor:Halpern-CATk-another}
 Let $v$ be an element of $X$,  
 and $\{y_n\}$ a sequence defined by $y_1\in X$ 
 and~\eqref{eq:Halpern-CATk}, 
 where $\{\alpha_n\}$ is a sequence in $\opclitvl{0}{1}$ 
 and $\{\lambda_n\}$ is a sequence of positive real numbers  
 such that 
 \begin{align*}
 \lim_{n\to \infty} \alpha_n = 0, 
 \quad  
 \sum_{n=1}^{\infty}\alpha_n^2=\infty, 
 \quad \textrm{and} \quad 
 \inf_n \lambda_n > 0. 
 \end{align*}
  If $\Argmin_X f$ is nonempty, 
 then both $\{y_n\}$ and $\{\tilde{R}_{\lambda_n f}y_n\}$ 
 are convergent to $Pv$, 
 where $P$ denotes the metric projection 
 of $X$ onto $\Argmin_X f$. 
\end{corollary}

\section*{Concluding remarks}

As we stated in Section~\ref{sec:intro}, 
it is known~\cite{MR3396425}*{Definition~4.1 and Lemma~4.2} 
that the classical resolvent given by~\eqref{eq:resolvent-Hadamard} 
is still well defined for any proper lower semicontinuous convex 
function in a complete $\CAT(1)$ space 
whose diameter is strictly less than $\pi/2$. 
However, it is also 
known~\cite{MR3463526}*{Corollary~3.3} 
that this diameter condition on the space 
implies that such a function always has a minimizer. 

On the other hand, according to~\cite{MR3463526}*{Theorem~4.2}, 
we can define another type of resolvent by~\eqref{eq:resolvent-CAT1-def2} with 
the perturbation function $\tan d \sin d$ 
in an admissible complete $\CAT(1)$ space.  
This makes it possible for us to study the existence 
of minimizers as well as the convergence to minimizers 
through the two proximal-type algorithms 
defined by~\eqref{eq:Mann-CAT1} and~\eqref{eq:Halpern-CAT1}. 

Finally, we point out that it is not clear whether 
there is any relationship between the two types of resolvents 
and hence we cannot deduce any result for the classical resolvents 
from the results obtained in this paper so far. 

\section*{Acknowledgment}
The authors would like to thank 
the anonymous referees for their helpful comments on the 
original version of this paper. 
This work was supported by JSPS KAKENHI Grant Numbers 
15K05007 and 17K05372.

\end{document}